\documentclass[11pt,reqno]{amsart}
\usepackage{amssymb,amsmath,amsthm,amsfonts}
\makeatletter
\@addtoreset{equation}{section}

\usepackage{xr}
\usepackage{graphicx}
\usepackage{graphics}
\usepackage{epsfig}
\usepackage{color}
\usepackage{tikz}
\usetikzlibrary{patterns}
\usepackage{array}
\usepackage{multicol}
\usepackage{stmaryrd}

\def\d{{\partial}}
\def\C{{\mathbb C}}
\def\H{\mathbb{H}}
\def\A{{\mathbb A}}
\def\Z{\mathbb {Z}}

\def\N{\mathbb {N}}

\def\RR{{\mathbb R}}
\def\R{{\RR}}

\def\O{\mathcal{O}}
\def\e{{\varepsilon}}
\def\trait (#1) (#2) (#3){\vrule width #1pt height #2pt depth #3pt}
\def\fin{\hfill\trait (0.1) (5) (0) \trait (5) (0.1) (0) \kern-5pt 
\trait (5) (5) (-4.9) \trait (0.1) (5) (0)}

\def\Re{{\rm Re\,}}
\newcommand{\ch}{{\mbox{\rm ch}\,}}
\newcommand{\sh}{{\mbox{\rm sh}\,}}
\newcommand{\argsh}{\mathop{\mbox{argsh}\,}}
\def\th{{\rm th\,}}
\def\div{{\rm\, div\,}}
\def\<{\left\langle}
\def\>{\right\rangle}

\newtheorem{prop}{Proposition}[section]
\newtheorem{thm}[prop]{Theorem}
\newtheorem{lem}[prop]{Lemma}
\newtheorem{cor}[prop]{Corollary}
\newtheorem{rem}[prop]{Remark}

\begin{document}
\title[\sf Decomposition theorem and Riesz basis for Axisymmetric potentials]{Decomposition theorem and Riesz basis for Axisymmetric potentials in the right half-plane.}
\author[S. CHAABI, S. RIGAT]{SLAH CHAABI, STEPHANE RIGAT}

\address{S. Chaabi, Aix-Marseille Universit\'e, I2M, UMR 7373, CMI, 39 rue Fr\'ed\'eric Joliot-Curie 13453 Marseille, France}
\email{slah.chaabi@univ-amu.fr}
\address{S. Rigat, Aix-Marseille Universit\'e, I2M, UMR 7373, CMI, 39 rue Fr\'ed\'eric Joliot-Curie 13453 Marseille, France}
\email{stephane.rigat@univ-amu.fr}
\date{\today}
\thanks{{\it Acknowledgements.} Both of the authors wish to thank Laurent Baratchart and Alexander Borichev for very useful discussions and remarks on a preliminary version of this paper.}
\begin{abstract} {The Weinstein equation with complex coefficients is the equation governing generalized axisymmetric potentials (GASP) which can be written as $L_m[u]=\Delta u+\left(m/x\right)\partial_x u =0$, where $m\in\mathbb{C}$. We generalize results known for $m\in\mathbb{R}$ to $m\in\mathbb{C}$. We give explicit expressions of fundamental solutions for Weinstein operators and their estimates near singularities, then we prove a Green's formula for GASP in the right half-plane $\mathbb{H}^+$ for Re $m<1$. We establish a new decomposition theorem for the GASP in any annular domains for $m\in\mathbb{C}$, which is in fact a generalization of the B\^ocher's decomposition theorem. In particular, using bipolar coordinates, we  prove for annuli that a family of solutions for GASP equation in terms of associated Legendre functions of first and second kind is  complete. For $m\in\mathbb{C}$, we show that this family is even a Riesz basis in some non-concentric circular annulus. }
\end{abstract}
\maketitle

\section{Introduction} 

In this article, we study the Weinstein differential operator
$$
L_m=\Delta+\frac m x\frac\d{\d x}
$$
with $m\in\C$, well-defined on the right half-plane $\H^+=\{(x,y)\in\R^2,\ x>0\}=\{z\in\C, \ {\rm Re}\,z>0\}$. This class of operators is also called operators governing axisymmetric potentials, they have been studied quite intensively in cases $m\in\N$ or $m\in\R$ in \cite{weinstein1948,weinstein1948ii,weinstein1948iii,weinstein1949,weinstein1951,weinstein1953,weinstein1954i,weinstein1954ii,weinstein1954iii,weinstein1955,weinstein1955i,weinstein1956,weinstein1960,weinstein1962,weinstein1965,weinacht1964,weinacht1965,weinacht1968,weinacht1974,copson1953,copson1970,copson1975,huber1953,huber1954,huber1956,erdelyi1956,erdelyi1965,erdelyi1965i,brelot1972,brelot1973,brelot1976,brelot1978,gilbert1962,gilbert1963,gilbert1963i,gilbert1964,liu2000}.
In this paper, we focus exclusively on case $m\in\C$ and some results for integer values of $m$ is recalled.
The Weinstein equation is written
\begin{equation}\label{WeinsteinEquation}
L_mu=0.
\end{equation}
In the sequel, the sense in which the solutions are studied will be specified.
We will also look at solutions to the equation in the sense of distributions
$$
L_mu=\delta_{(x,y)},
$$
where $\delta_{(x,y)}$ denotes the Dirac mass at $(x,y)\in\H^+$.

Here, we deliberately restrict to the case of dimension 2, but many results can be extended directly to the case of higher dimension, {\it i.e.} for the operators
$$
\sum_{k=1}^n\d_{x_i^2}+{m\over x_n}{\d_{x_n}}
$$
on the half-space $\H_n^+=\{x\in\R^n,\ x_n>0\}$.
%

This is Weinstein who was the first to introduce this class of operators in 1948 in \cite{weinstein1948iii}, he studied the case where $m\in\N^*$. In particular, he obtained the following mean-value formula for axisymmetric  potentials which can be extended continuously to $\overline{\H}_+$
$$
u(0,0)\int_{-{\pi\over2}}^{\pi\over2}\sin^m\theta d\theta=\int_{-{\pi\over2}}^{\pi\over2}u(re^{i\theta})\sin^m\theta d\theta\eqno{(PM)}
$$ 
and he first gave an expression of a fundamental solution in terms of Bessel functions and then in terms of elliptic integrals where the singular point is taken on the $y$-axis. He also established the link between the axisymmetric potentials for $m\in\N^*$ and the harmonic functions on $\R^{m+2}$, that we will recall in the proposition
\ref{Weinstein2}.

In \cite{weinstein1951,weinstein1953,weinstein1954}, Weinstein and Diaz-Weinstein established the correspondence principle that we will recall between axisymmetric potentials corresponding to $m$ and those corresponding to $2-m$ (proposition \ref{PrincipeWeinstein}). They deduced an expression of a fundamental solution (where the singular point is taken on the $y$- axis) for $m\in\R$ and they made a link between the Weinstein equation  and Tricomi equations and their fundamental solutions.

In \cite{huber1954}, Huber obtained a Poisson-type formula which generalizes {(PM)}. He also dealt with the extensions of GASP in the whole complex plane. He was also interested in removal properties of singularities for GASP in \cite{huber1956}, and he gave an abstract representation formula of nonnegative GASP that Brelot generalized in \cite{brelot1972,brelot1978}.

Moreover, Vekua gave means to express fundamental solutions of elliptic equations with analytic coefficients by using the Riemann functions, introduced in the past (see eg \cite{garabedian}) in the real hyperbolic context, he generalized to elliptic equations through the complex operators $\d_z$ et $\d_{\overline z}$ in \cite{vekua}.
In heuristic words, in the same way that we can say a harmonic function is the real part of a holomorphic function, or the sum of a holomorphic and an anti-holomorphic function, Vekua expressed the fact that solutions of elliptic equations, and therefore especially GASP, are written as a sum of two functional, one applied to an arbitrary holomorphic function and functiun applied to an anti-holomorphic functiun also arbitrary. The functional write can explicitly in terms of Riemann function, which are obtained by using the hypergeometric functions (\cite{vekua}) or using fractional derivations (\cite{copson1953}).
In \cite{henrici1957}, Henrici gave a very interesting introduction to the work of Vekua.

More recently, by using the work of Vekua in \cite{savina2007}, Savina gave a series representation of fundamental solutions for the operator $\hat{L}u=\Delta u+a\d_x u+b\d_y u+cu$ and she studied the convergence of this series. She gave an application of the Helmholtz equation.

In \cite{gilbert1963i}, Gilbert studied the non-homogeneous Weinstein equation $m\geq0$, he gave an integral formula for this class of equations, in particular,  an explicit solution is given  when the second member depends only of one variable.

Some Dirichlet problems can be found in  \cite{liu2000} in \cite{fokas2011} in special geometry ("geometry with separable variable").

Even if some results presented in this paper are known for real values of $m$, we make a totally  self-contained presentation with elementary technics not usually used in the previous quoted papers. For instance, usual arguments involving estimates of hypergeometric integrals are replaced by arguments using Lebesgue dominated convergence theorem. The main result is a decomposition theorem for axisymmetric potentials which is new also for real values of $m$. We obtain a Liouville-type result for the solutions of Weinstein equation on $\H^+$, the interesting side of this result is the fact that there is a lost of strict ellipticity of the Weinstein operator on the boundary of $\H^+$. An application of the decomposition theorem is given by showing that an explicit family of axisymmetric potentials constructed with the introduction of bipolar coordinates is a Riesz basis in a some annulus. {\color{white}{\cite{chaabi}}}

\section{Notations and preliminaries}
\label{sec:notations_generales}

Throughout the following, $\H^+=\{(x,y)\in\R^2,\ x>0\}$ will denote the right half-plane.
All scalar functions will be complex valued. If $\Omega$ is an open set of $\R^n$ with $n\in\N^*$, $\mathcal{D}(\Omega)$ will designate the space of $C^\infty$ functions compactly supported on $\Omega$ and the support of an arbitrary function $f$ defined on $\Omega$ is ${\rm supp }\,f:=\overline{\left\{x\in\Omega,\ f(x)\neq0\right\}}$.

Let $K$ be a compact set of $\Omega$, $\mathcal{D}_K(\Omega)$ is the set of functions $\varphi\in\mathcal{D}(\Omega)$ such that supp $\varphi\subset K$.

The partial derivatives of a differentiable function $u$ on an open set $\Omega\subset\R^n$ will be denoted $\d u\over \d x_i$ or $\d_{x_i}u$, or sometimes $u_{x_i}$ with $i\in\llbracket1,n\rrbracket$ (for $a<b\in\N$, $\llbracket a,b\rrbracket$ denotes the set of all integers between $a$ and $b$).

If $\alpha=(\alpha_1,\ldots,\alpha_n)\in\N^n$ is a multi-index, we will denote
$$
\d^\alpha:=\d_{x_1}^{\alpha_1}\cdots\d_{x_n}^{\alpha_n}={\d^{|\alpha|}\over\d x_1^{\alpha_1}\cdots\d x_n^{\alpha_n}}
$$
with $|\alpha|=\alpha_1+\cdots+\alpha_n$.

It is assumed that the reader is familiar with the terminology of distributions and we refer to \cite{hormander}.

Let $L$ be a linear differential operator on $\Omega$,
$$
L=\sum_{|\alpha|\leq N}a_{\alpha}\d^\alpha
$$
where $N\in\N$, the previous summation is performed on the multi-indices $\alpha$ of length $|\alpha|\leq N$, $a_\alpha$ are $C^\infty(\Omega)$ functions.

By definition, if $T$ is a distribution, $LT$ will be the distribution : $LT=\sum_{|\alpha|\leq N}a_\alpha\d^\alpha T$.

$L^\star$ will designate the adjoint operator of $L$ in the sense of distributions, namely if $T$ is a distributions,
$$
L^\star T=\sum_{|\alpha|\leq N}(-1)^{|\alpha|}\d^\alpha(a_\alpha T).
$$
It is noticed that if $f,\,g\in\mathcal{D}(\Omega)$, we have
$$
\<Lf,g\>=\<f,L^\star g\>.
$$
Let $a\in\Omega$ and $L$ be a differential operator on $\Omega$. {\it A fundamental solution of $L$ on $\Omega$ at $a\in\Omega$} is a distribution $T_a$ such that
$$
LT_a=\delta_a,
$$
where the previous equality is taken in the sense of distributions on $\Omega$.

This equality can be rewritten
$$
\forall\varphi\in\mathcal{D}(\Omega),\qquad\varphi(a)=\<LT_a,\varphi\>=\<T_a,L^\star\varphi\>.
$$
In particular, if $a\in\Omega$ and if $T_a$ is a fundamental solution of $L^\star$ at $a$ on $\Omega$ and if $g\in\mathcal{D}(\Omega)$ is such that
$g=L(\varphi)$ with $\varphi\in\mathcal{D}(\Omega)$, then
$$
\forall a\in\Omega,\qquad\varphi(a)=\<T_a,g\>.
$$
Indeed, we have
$$
\forall a\in\Omega,\qquad\varphi(a)=\<\delta_a,\varphi\>=\<L^*Ta-,\varphi\>=\<T_a,L\varphi\>=\<T_a,g\>.
$$
These fundamental solutions is therefore a good tool for solving  $L\varphi=g$ on $\mathcal{D}(\Omega)$ if $g\in\mathcal{D}(\Omega)$.

\bigskip
If $m\in\N^*$, the Laplacian in $\R^m$ will be denoted $\Delta_m$, or $\Delta$ when $m=2$.
  For $m\in\C$, $L_m$ denotes the {\it Weinstein operator} :  $\forall (x,y)\in\H^+$,
$$
L_m u(x,y) =\Delta u(x,y)+\frac{m}{ x}\frac{\d u}{\d x}(x,y),\quad \text{ where }\quad u\in {C}^2(\H^+).
$$
The following notation will be sometimes used : if $f(x,y)=(f_1(x,y),f_2(x,y))$ is a ${C}^1$ vector function on an open set of $\R^2$ and $\C^2$ valued, then
$$
{\rm div}\,(f) :={\d f_1\over\d x}+{\d f_2\over\d y}.
$$

Similarly, if $f:\R^2\to\C$ is a ${C^1}$ scalar function on an open set of $\R^2$ and $\C$ valued, then
$$
\nabla f:=\left(\frac{\d f}{\d x},\frac{\d f}{\d y}\right).
$$

With these notations, the operator $L_m$ can be written as follows : if $u\in {C}^2(\H^+)$, then
$$
L_m u(x,y)=x^{-m}{\rm div}(x^m\nabla u)(x,y).
$$

It is clear from the Schwarz rule that if  $u$ is a function defined on a connected open set of $\H^+$ such that ${\rm div}(\sigma\nabla u)=0$ where $\sigma:\H^+\to\R^+_*$ is a ${C}^1$ function, then there is a function $v$ which satisfies the well-known generalized Cauchy-Riemann system of equations :
$$
\left\{
\begin{matrix}
\displaystyle{\d v\over\d x}=-\sigma{\d u\over\d y}\cr
\cr
\displaystyle{\d v\over\d y}=\sigma{\d u\over\d x}
\end{matrix}
\right.
$$
and $v$ satisfies the conjugate equation ${\rm div}(\frac1\sigma\nabla v)=0$ (see for exemple \cite{BLRR}). 
This observation justifies the fact that we call $L_{-m}$ with $m\in\C$ the conjugate operator of $L_m$. 

 $L_m^\star$ denotes adjoint operator of $L_m$ : for all $u\in {C}^2(\H^+)$ and for all $(x,y)\in\H^+$,
 $$
L_m^\star u(x,y)=\Delta u(x,y)-{\d \over\d x}\left({m \,u(x,y)\over x}\right)=
\Delta u(x,y)-{m\over x}{\d u\over \d x}(x,y)+{m\over x^2}u(x,y)
$$
This definition is given on $\H^+$ but it is easily transposed to the case of an open set $\Omega$ of $\H^+$.

In the case where the functions involved do not depend only of  $x$ and $y$, we will write $L_{m,x, y}$ instead of $L_{m}$, which means that the partial derivatives are related to the variables $x$ and $y$, and all other variables are considered to be fixed.

 \noindent
If $u\in\mathcal{D}(\H^+)$, we define  $S_m u\in\mathcal{D}(\H^+)$ by
$$
(S_m u)(x,y)=x^{-m}u(x,y).
$$
\noindent
If $u\in\mathcal{D}(\H^+)$, we define $Du\in\mathcal{D}(\H^+)$ by
$$
(Du)(x,y)={\d u\over\d x}(x,y).
$$

These operators satisfy the following proposition :

\noindent
\begin{prop}\label{conjugaison}
 $S_m$ conjugates $L_m^\star$ and $L_m$, $D$ conjugates $L_{-m}^\star$ and $L_m$, which means that
$$
S_mL_m^\star=L_mS_m,\quad  L_{-m}^\star D=DL_m.$$
\end{prop}
\noindent
\begin{proof}
\noindent
%
%
%
Straightforward computations.
\end{proof}

\bigskip\noindent
\begin{rem}
\begin{enumerate}
\item Let $m\in\C$, $S_m$ and $L_m S_m$ are auto-adjoints operators, {ie.} $S_m=S_m^\star$ and $L_m S_m=(L_m S_m)^\star$.
\item There is a result, which generalizes the first point of this remark about the conjugation of operators $L_m$ and $L_m^\star$. 

Let $\sigma\, :\, \Omega\to\C$ be a ${C}^1$ function which does not vanish, the operator defined on ${C}^2(\Omega)$ by :  for $u\in {C}^2(\Omega)$,
$$
P_\sigma u(x,y)={1\over\sigma(x,y)}{\rm div}\left(\sigma(x,y)\nabla u(x,y)\right),
$$
where $\Omega$ is an open set of $\R^2$.

Then
$$
P_\sigma^\star={\rm div}\left(\sigma\nabla\left({\cdot\over\sigma}\right)\right).
$$
Indeed, if $u,v\in\mathcal{D}(\Omega)$, then we have by using the derivation in the sense of distributions,
\begin{align*}
\langle P_{\sigma}u,v\rangle&=\displaystyle \int_{\Omega}{1\over\sigma(x,y)}{\rm div}\left(\sigma(x,y)\nabla u(x,y)\right)v(x,y)\,dxdy\cr
&=\displaystyle -\int_{\Omega}\sigma\nabla u\cdot\nabla\left({v\over\sigma}\right)\,dxdy\cr
&=\displaystyle \int_{\Omega}u\,{\rm div}\left(\sigma\nabla\left({v\over\sigma}\right)\right)\cr
&=\displaystyle    \langle u,P_{\sigma}^\star v\rangle
\end{align*}

We define $S_\sigma$ the operator such that for $u\in {C}^2(\Omega)$,
$$
(S_\sigma u)(x,y)={1\over \sigma(x,y)}u(x,y).
$$
Thus, $S_\sigma$ conjugates $P_\sigma$ and $P_\sigma^\star$, where $P_\sigma^\star={\rm div}\left(\sigma\nabla\left({\cdot\over\sigma}\right)\right)$ because, in an obviously way, we have $S_\sigma P_\sigma^\star=P_\sigma S_\sigma$.
\end{enumerate}
\end{rem}

If $m$ is a positive integer, we introduce the operator $T_m : u\mapsto v$ defined as follows :

for a function $u$ defined on an open set $\Omega$ of $\H^+$, the function $v$ is defined on $\{x\in\R^{m+2},\ (\sqrt{x_1^2+\cdots+x_{m+1}^2},x_{m+2})\in \Omega\}$ by 
$$
v(x_1,\ldots,x_{m+2})=u(\sqrt{x_1^2+\cdots+x_{m+1}^2},x_{m+2}).
$$

The two following propositions can be found in Weinstein work (\cite{weinstein1953}) and will be useful in the sequel 
(the (short) proofs are just direct computations):

\begin{prop}{\bf (Weinstein principle \cite{weinstein1953})}\label{PrincipeWeinstein}
Let $\Omega$ be a relatively compact open set of $\H^+$, if $u:\Omega\to\C$ is ${C}^2$, then for all $m\in\C$,
$$
L_mu=x^{1-m}L_{2-m}[x^{m-1}u].
$$
\end{prop}
%
%

\begin{prop} {\bf (\cite{weinstein1948iii})} \label{Weinstein2}
Let $\Omega$ be a relatively compact open set of $\H^+$.  For $u\in{C}^2(\Omega)$and if  $m\in\N$, then $\Delta_{m+2}(T_mu)=T_m(L_mu)$.
\end{prop}


The two previous propositions will allow us to calculate fundamental solutions for $L_m$ and $L_m^\star$ with $m\in\N$ in a first step, and thereafter, for $m\in\Z$. Finally, estimates of these expressions will show that the expressions obtained actually provide fundamental solutions of $L_m$ and $L_m^\star$ with $m\in\C$.

\section{Integral expressions of fundamental solutions for integer values of $m$.}
We recall the definition of the Dirac mass in a point : if $(x,y)\in\R^2$, $\delta_{(x,y)}$ is the distribution defined by
$$
\forall \varphi\in\mathcal{D}(\R^2),\qquad\langle \delta_{(x,y)},\varphi\rangle=\varphi(x,y).
$$
Let $m$ be a positive integer.
\begin{prop}{\bf (partially in \cite{weinstein1954, weinacht1968,weinstein1948iii})}\label{WeinsteinSolFond}
Let $m\in\N^\star$. For $(x,y)\in\H^+$ and $(\xi,\eta)\in\H^+$,
$$
E_m(x,y,\xi,\eta)=-{\xi^m\over2\pi  }\int_{\theta=0}^{\pi}{\sin^{m-1}\theta \, d\theta \over \left[(x-\xi)^2+4x\xi\sin^2\left({\theta\over 2}\right)+(y-\eta)^2\right]^{m/2}}
$$
is a fundamental solution on $\H^+$ for the operator $L^\star_{m,\xi,\eta}$ at the fixed point $(x,y)\in\H^+$, which means that in the sense of distributions, we have $\H^+$ :
$$
L_{m,\xi,\eta}^\star E_m(x,y,\xi,\eta)=\delta_{(x,y)}(\xi,\eta).
$$
Moreover, if $(\xi,\eta)\in\H^+$ is fixed, then in the sense of distributions on $\H^+$, we have
$$
L_{m,x,y}E_m(x,y,\xi,\eta)=\delta_{(\xi,\eta)}(x,y),
$$
which means that $E_m$ is a fundamental solution on $\H^+$ of the operator $L_{m,x,y}$ at the fixed point $(\xi,\eta)\in\H^+$.
\end{prop}

\begin{proof}
Let $m\in\N^*$. We recall that $$
E(x)=-{1\over m\,\omega_{m+2} \|x\|^m}, \quad x\in \R^{m+2},
$$
is a fundamental solution for the Laplacian on $\R^{m+2}$ 
{\it i. e.} that in the sense of distributions, $\Delta_{m+2} E=\delta_0$, where $\omega_{m+2}$ is the area of the unit sphere $\R^{m+2}$. Thus, for all $v\in\mathcal{D}(\R^{m+2})$,
$$
v(t_1,...,t_{m+2})=-{1\over m\,\omega_{m+2}}\int_{\tau\in\R^{m+2}}\Delta_{m+2}v(\tau){d\tau_1 d\tau_2 ... d\tau_{m+2}\over \left((\tau_1-t_1)^2+\cdots+(\tau_{m+2}-t_{m+2})^2\right)^{m/2}}
$$
where $\tau=(\tau_1,...,\tau_{m+2})$.

\noindent
Applying this relation to the function $v=T_mu$ where $u\in\mathcal{D}(\H^+)$ and due to the proposition \ref{Weinstein2}, we have for all $(x,y)\in\H^+$,
$$
u(x,y)=-{1\over m\,\omega_{m+2}}\int_{\R^{m+2}}{(L_m u)(\sqrt{\xi_1^2+\cdots+xi_{m+1}^2},\xi_{m+2})\,d\xi_1\cdots d\xi_{m+2}\over \left((\xi_1-x)^2+\xi_2^2+\cdots+\xi_{m+1}^2+(\xi_{m+2}-y)^2 \right)^{m/2}}
$$

We will simplify this integral expression. For this, we will consider the following  hyper-spherical coordinates :

\begin{align*}
\xi_1&=\xi\cos\theta_1\cr
\xi_2&=\xi\sin\theta_1\cos\theta_2\cr
\vdots\,&=\quad\vdots\cr
\xi_{m-1}&=\xi\sin\theta_1\cdots\sin\theta_{m-2}\cos\theta_{m-1}\cr
\xi_{m}&=\xi\sin\theta_1\cdots\sin\theta_{m-1}\cos\theta_m\cr
\xi_{m+1}&=\xi\sin\theta_1\cdots\sin\theta_{m}
\end{align*}
where $\xi^2=\xi_1^2+\cdots+\xi_{m+1}^2$, $\theta_m\in\left]-\pi,\pi\right[$ and $\theta_1,\ldots,\theta_{m-1}\in\left]0,\pi\right[$. The absolute value of the determinant of the Jacobian matrix defined by this system of coordinates is
$$
\xi^m\sin\theta_{m-1}\sin^{2}\theta_{m-2}\cdots\sin^{m-1}\theta_1
$$
Then we have for all $(x,y)\in\H^+$,
\begin{equation}\label{SoFon}
u(x,y)=\int_{\eta=-\infty}^\infty \int_{\xi=0}^\infty L_m(u)(\xi,\eta)E_m(x,y,\xi,\eta)d\xi d\eta
\end{equation}

with
$$
E_m(x,y,\xi,\eta)=-{ \xi^m \over m\,\omega_{m+2}}\int_{\theta_m=-\pi}^\pi \int_{\theta_1,\ldots,\theta_{m-1}=0}^\pi {\sin\theta_{m-1}\sin^{2}\theta_{m-2}\cdots\sin^{m-1}\theta_1 \, d\theta_1 \ldots d\theta_{m} \over \left(\xi^2-2x\xi\cos\theta_1+x^2+(y-\eta)^2\right)^{m/2}}
$$
Since $I:= \int_{\theta_m=-\pi}^{\pi}\int_{\theta_2,\ldots,\theta_{m-1}=0}^\pi \sin\theta_{m-1}\sin^{2}\theta_{m-2}\cdots\sin^{m-2}\theta_2 \, d\theta_2 \ldots d\theta_{m-1}d\theta_m$ is the area of the unit sphere on $\R^{m}$ because
$$
\omega_m=\int_{\mathbb{S}_{m}} 1\,d\sigma=\int_{\theta_{m-1}=-\pi}^{\pi} \int_{\theta_1,\ldots,\theta_{m-2}=0}^\pi \sin\theta_{m-2}\sin^{2}\theta_{m-3}\cdots\sin^{m-2}\theta_1 \, d\theta_2 \ldots d\theta_{m-1},
$$
$E_m$ can be written as :
$$
E_m(x,y,\xi,\eta)=-{\omega_{m }\xi^m \over m\,\omega_{m+2}}\int_{\theta=0}^{\pi}{\sin^{m-1}\theta \, d\theta \over \left(\xi^2-2x\xi\cos\theta+x^2+(y-\eta)^2\right)^{m/2}}.
$$
Using the fact that $\omega_m={2\pi^{{m\over2}}\over\Gamma\left({m\over2}\right)}$, we get
$$
E_m(x,y,\xi,\eta)=-{ \xi^m\over2\pi  }\int_{\theta=0}^{\pi}{\sin^{m-1}\theta \, d\theta \over \left((x-\xi)^2+4x\xi\sin^2\left({\theta\over 2}\right)+(y-\eta)^2\right)^{m/2}}
$$
and we have proved moreover thanks to (\ref{SoFon}) that 
$$
L_{m,\xi,\eta}^*E_m(x,y\xi,\eta)=\delta_{(x,y)}(\xi,\eta).
$$
Moreover, since for all $(x,y)\in\H^+$ and for all $(\xi,\eta)\in\H^+$, we have
$$
E_m(x,y,\xi,\eta)=\left({ x\over\xi}\right)^{-m} E_m(\xi,\eta,x,y)
$$
and thanks to the Proposition \ref{conjugaison},  $S_m$ conjugates $L_m^\star$ and $L_m$, we have in the sense of distributions
$$
L_{m,x,y}E_m(x,y,\xi,\eta)=L_{m,x,y}\left(\left({ x\over\xi}\right)^{-m} E_m(\xi,\eta,x,y)\right)=\left({ x\over\xi}\right)^{-m} L^\star_{m,x,y}E_m(\xi,\eta,x,y),
$$
then
$$
L_{m,x,y}E_m(x,y,\xi,\eta)=\left({ x\over\xi}\right)^{-m} \delta_{(\xi,\eta)}(x,y)= \delta_{(\xi,\eta)},
$$
and this completes the proof.

\end{proof}

For $m\in\Z\setminus \N$, the previous proposition and the Weinstein principle gives us the following proposition :

\begin{prop}{\bf (partially in \cite{weinstein1954, weinacht1968,weinstein1948iii})}\label{Mnegatif}
Let $m\in\Z\setminus\N^*$. For $(x,y)\in\H^+$ and $(\xi,\eta)\in\H^+$,
$$
E_{m}(x,y,\xi,\eta)=\left({\xi\over x}\right)^{m-1} E_{2-m}(x,y,\xi,\eta)
$$
$$
=-{\xi x^{1-m}\over 2\pi}\int_{\theta=0}^{\pi}{\sin^{1-m}\theta \, d\theta \over \left[(x-\xi)^2+4x\xi\sin^2\left({\theta\over 2}\right)+(y-\eta)^2\right]^{1-{m\over2}}}
$$
is a fundamental solution on $\H^+$ for the operator $L^\star_{m,\xi,\eta}$ at the fixed point $(x,y)\in\H^+$ and it is also a fundamental solution on $\H^+$ of the operator $L_{m,x,y}$ at the fixed point $(\xi,\eta)\in\H^+$.
\end{prop} 

\begin{proof}

We have for all $m\in\N^*$, $u\in\mathcal{D}(\H^+)$ and $(x,y)\in\H^+$,
$$
u(x,y)=\int_{(\xi,\eta)\in\H^+}(L_m u)E_m(x,y,\xi,\eta)\,d\xi d\eta,
$$
and by the Weinstein principle (proposition \ref{PrincipeWeinstein}), we have
$$
u(x,y)=\int_{\H^+} \xi^{1-m}L_{2-m} (\xi^{m-1}u)E_m(x,y,\xi,\eta)\,d\xi d\eta.
$$
Denoting $v(x,y)=x^{m-1}u(x,y)$, we obtain
$$
x^{1-m}v(x,y)=\int_{\H^+} \xi^{1-m}(L_{2-m}v)E_m(x,y,\xi,\eta)\,d\xi d\eta,
$$
then, for all $m'\in\Z\setminus\N^*$, $v\in\mathcal{D}(\H^+)$ and $(x,y)\in\H^+$, putting $m=2-m'$, we have
$$
v(x,y)=\int_{\H^+} (L_{m'}v)\left({\xi\over x}\right)^{m'-1} E_{2-m'}(x,y,\xi,\eta)\,d\xi d\eta.
$$
The proof of the second point is totally similar.
\end{proof}

\section{Fundamental solutions for the Weinstein equation with complex coefficients}

In this section, we will generalize the result obtained in the previous section for $m\in\Z$ to $m\in\C$.

More precisely, if Re $m\geq 1$, then
$$
E_m=-\frac{\xi^m}{2\pi}\int_{\theta=0}^{\pi}\frac{\sin^{m-1}\theta\,d\theta}{[(x-\xi)^2+4x\xi\sin^2\left(\frac\theta2\right)+(y-\eta)^2]^{m/2}}
$$
is suitable, and if Re $m<1$, then 
$$
E_m=-{\xi x^{1-m}\over 2\pi}\int_{\theta=0}^{\pi}{\sin^{1-m}\theta \, d\theta \over \left[(x-\xi)^2+4x\xi\sin^2\left({\theta\over 2}\right)+(y-\eta)^2\right]^{1-{m\over2}}}
$$
is suitable.

\bigskip\noindent

In the sequel, $E_m$ will always designate the corresponding formula (depending of Re $m\geq 1$ or Re $m<1$).

\begin{prop}
For $m\in\C$ and $(\xi,\eta)\in\H^+$ fixed, we have
$$
\forall (x,y)\in\H^+\setminus\{(\xi,\eta)\}\qquad L_{m,x,y}E_m(x,y,\xi,\eta)=0.
$$
and for $(x,y)\in\H^+$ fixed, we have
$$
\forall (\xi,\eta)\in\H^+\setminus\{(x,y)\}\qquad L^\star_{m,\xi,\eta}E_m(x,y,\xi,\eta)=0.
$$
\end{prop}

\begin{proof} 
For convenience in the calculations, it should be denoted
$$
f_m(x,y,\xi,\eta,\theta)={1\over[(x-\xi)^2+4x\xi\sin^2\left(\frac\theta2\right)+(y-\eta)^2]^{{m\over2}}}.
$$ 
To prove the first equality of the proposition, it suffices to show that
$$
\int_{\theta=0}^\pi L_{m,x,y} f_m(x,y,\xi,\eta,\theta) \sin^{m-1}\theta\,d\theta=0.
$$
Let's compute the derivatives of the function $f_m$ :
$$
\d_x f_m={-m\over2}{2(x-\xi)+4\xi\sin^2{\theta\over2}\over [(x-\xi)^2+4x\xi\sin^2\left(\frac\theta2\right)+(y-\eta)^2]^{{m\over2}+1}}\quad\left(=-m(x-\xi\cos\theta)f_{m+2}\right)
$$
and
$$
\d_{xx} f_m={-m\over [(x-\xi)^2+4x\xi\sin^2\left(\frac\theta2\right)+(y-\eta)^2]^{{m\over2}+1}}+\qquad\qquad\qquad\qquad$$
$$
\qquad\qquad\qquad+{m\over2}\left({m\over2}+1\right){(2(x-\xi)+4\xi\sin^2{\theta\over2})^2\over [(x-\xi)^2+4x\xi\sin^2\left(\frac\theta2\right)+(y-\eta)^2]^{{m\over2}+2}}
$$
and
$$
\d_{yy} f_m={-m\over [(x-\xi)^2+4x\xi\sin^2\left(\frac\theta2\right)+(y-\eta)^2]^{{m\over2}+1}}+\qquad\qquad\qquad\qquad$$
$$
\qquad\qquad\qquad+{m\over2}\left({m\over2}+1\right){(2(y-\eta))^2\over [(x-\xi)^2+4x\xi\sin^2\left(\frac\theta2\right)+(y-\eta)^2]^{{m\over2}+2}}
$$
We then have
$$
\Delta f_m={-2m\over [(x-\xi)^2+4x\xi\sin^2\left(\frac\theta2\right)+(y-\eta)^2]^{{m\over2}+1}}+\qquad\qquad\qquad\qquad
$$
$$
\qquad\qquad\qquad\qquad+{m\over2}\left({m\over2}+1\right){(2(x-\xi)+4\xi\sin^2{\theta\over2})^2+(2(y-\eta))^2\over [(x-\xi)^2+4x\xi\sin^2\left(\frac\theta2\right)+(y-\eta)^2]^{{m\over2}+2}}.
$$
However
$$
\left(2(x-\xi)+4\xi\sin^2{\theta\over2}\right)^2+\left(2(y-\eta)\right)^2=4\left[(x-\xi)^2+4x\xi\sin^2\left(\frac\theta2\right)+(y-\eta)^2\right]-4\xi^2\sin^2\theta
$$
then
$$
\Delta f_m={m^2\over [(x-\xi)^2+4x\xi\sin^2\left(\frac\theta2\right)+(y-\eta)^2]^{{m\over2}+1}}\qquad\qquad\qquad
$$
$$
\qquad\qquad\qquad-{m(m+2)\xi^2 \sin^2\theta\over [(x-\xi)^2+4x\xi\sin^2\left(\frac\theta2\right)+(y-\eta)^2]^{{m\over2}+2}}\, .
$$
Noting that
$$
{\d f_{m+2}\over \d\theta}=-(m+2){x\xi\sin\theta\over  [(x-\xi)^2+4x\xi\sin^2\left(\frac\theta2\right)+(y-\eta)^2]^{{m\over2}+2}}\, ,
$$
we have
$$
\Delta f_m=m^2 f_{m+2}+m{\xi\over x}\sin\theta {\d f_{m+2}\over \d\theta}
$$
and by integration by parts, we have :
$$
\int_{\theta=0}^\pi \Delta f_m \sin^{m-1}\theta\,d\theta=m^2\int_{\theta=0}^\pi f_{m+2}\sin^{m-1}\theta\,d\theta+m{\xi\over x}\int_{\theta=0}^\pi {\d f_{m+2}\over \d\theta}\sin^{m}\theta\,d\theta
$$
$$
={m\over x}\int_{\theta=0}^\pi   m\left(x-\xi\cos\theta\right) f_{m+2}\sin^{m-1}\theta \,d\theta
$$
$$
=-{m\over x}\int_{\theta=0}^\pi \d_x f_m \sin^{m-1}\theta\,d\theta,
$$
and the result is deduced in the case Re $m\geq1$. The proof is totally similar if Re $m<1$. The second equality of the proposition can be deduced immediately of the fact that $S_m$ conjugates $L_m^\star$ and $L_m$ (see proposition \ref{conjugaison}).
\end{proof}

In the sequel, we will denote
$$
d^2=(x-\xi)^2+(y-\eta)^2\hbox{ and }k=\frac{4x\xi}{d^2}.
$$

The following proposition gives the behavior of these functions near their singularity. And it will be useful to show that they are indeed fundamental solutions for $m\in\C$  and not only  for integer values of $m$. 
In particular, we show that the behavior of the fundamental solutions is close to the behavior of fundamental solutions for the Laplacian. This fact is well known for elliptic operators. But we emphasize here that in the proof of this proposition, the estimates of elliptic integrals are totally elementary estimates (using the dominated convergence theorem) and here we do not use estimates arising from classical estimates of hypergeometric functions.
From those integral expressions, we deduce the following estimations :
\begin{prop}\label{estimation}
Let $m\in\C$. For $(x,y)\in\H^+$ fixed,
$$
E_m(x,y,\xi,\eta)\mathop{\sim}_{(\xi,\eta)\to(x,y)}{1\over 2\pi}\ln\sqrt{(x-\xi)^2+(y-\eta)^2}
$$
\end{prop}

\begin{proof}
We start with Re $m\geq1$.

In this case, we have : 

$$
E_m(x,y,\xi,\eta)=-\frac{\xi^m}{2\pi}\int_{\theta=0}^{\pi}\frac{\sin^{m-1}\theta \, d\theta}{\left[(x-\xi)^2+4x\xi\sin^2\left(\frac{\theta} 2\right)+(y-\eta)^2\right]^{m/2}}
$$
$$
=-\frac1{2\pi}\left(\frac{\xi}d\right)^m\int_{\theta=0}^\pi\frac{\sin^{m-1}\theta\,d\theta}{(1+k\sin^2\frac\theta2)^{m/2}}.
$$
Note that when $d\to0$, $k\to+\infty$.

We have the following proposition : 

\begin{prop}\label{Equivalent1}
When $k\to+\infty$ and $m\in\C$,
$$
\int_{\theta=0}^\pi\frac{\sin^{m-1}\theta\,d\theta}{(1+k\sin^2\frac\theta2)^{m/2}}\mathop\sim_{k\to+\infty}\frac{2^{m-1}}{k^{m/2}}\ln k.
$$
\end{prop}

\begin{proof}

Putting $u=\sin\frac\theta2$, this integral is equal to
$$
2^m\int_0^1\frac{u^{m-1}(1-u^2)^{\frac{m-2}2}du}{(1+ku^2)^{m/2}}={2^m\over k^{m/2}}\int_0^1\frac{u^{m-1}(1-u^2)^{\frac{m-2}2}du}{(\frac1k+u^2)^{m/2}}.
$$
However
$$
\int_0^1\frac{u^{m-1}(1-u^2)^{\frac{m-2}2}du}{(\frac1k+u^2)^{m/2}}-\int_0^1\frac{u^{m-1}du}{(\frac1k+u^2)^{m/2}}=
-\int_0^1\frac{u^{m-1}}{\left(\frac1k+u^2\right)^{m/2}}(1-(1-u^2)^{\frac{m-2}2})du
$$
and by monotone convergence, we obtain
$$
\mathop{\longrightarrow}_{k\to+\infty}-\int_0^1\frac{u^{m-1}}{\left(u^2\right)^{m/2}}(1-(1-u^2)^{\frac{m-2}2})du=-\int_0^1\frac{1-(1-u^2)^{\frac{m-2}2}}udu
$$
The change of variable $u=\frac1{\sqrt k}\sh t$ gives us
$$
\int_0^1\frac{u^{m-1}du}{(\frac1k+u^2)^{m/2}}=\int_0^{\argsh\sqrt k}\th^{m-1} t\,dt
$$
Since $\th^{m-1}t$ tends to 1 when $t\to+\infty$, we deduce that when $k\to+\infty$
$$
\int_0^{\argsh\sqrt k}\th^{m-1}dt\mathop{\sim}_{k\to+\infty}\int_0^{\argsh\sqrt k}dt=\argsh\sqrt k\mathop{\sim}_{k\to+\infty}\frac12\ln k.
$$
The proof is completed.
\end{proof}

Due to Proposition \ref{Equivalent1}, we have
$$
E_m(x,y,\xi,\eta)\mathop{\sim}_{d\to0+}-{1\over 2\pi}\left(\frac xd\right)^m\frac{2^{m-1}}{k^{m/2}}\ln k\mathop\sim_{d\to0+}\frac1{2\pi}\ln d.
$$

The case Re $m<1$ is analogous.

\end{proof}
Now, we can prove the main result of this section,
\begin{thm}\label{theo}
Let $m\in\C$. For $(x,y)\in\H^+$ and $(\xi,\eta)\in\H^+$,
$$
E_m(x,y,\xi,\eta)=-{\xi^m\over2\pi  }\int_{\theta=0}^{\pi}{\sin^{m-1}\theta \, d\theta \over \left[(x-\xi)^2+4x\xi\sin^2\left({\theta\over 2}\right)+(y-\eta)^2\right]^{m/2}} \qquad\hbox{if Re $m\geq 1$}
$$
$$
\hbox{and }
E_{m}(x,y,\xi,\eta)=\left({\xi\over x}\right)^{m-1} E_{2-m}(x,y,\xi,\eta)\qquad\qquad\qquad\qquad\qquad\qquad\qquad\qquad
$$
$$\qquad\qquad
=-{\xi x^{1-m}\over 2\pi}\int_{\theta=0}^{\pi}{\sin^{1-m}\theta \, d\theta \over \left[(x-\xi)^2+4x\xi\sin^2\left({\theta\over 2}\right)+(y-\eta)^2\right]^{1-{m\over2}}}\qquad\hbox{if Re $m<1$}
$$
is a fundamental solution on $\H^+$ for $L^\star_{m,\xi,\eta}$ at the fixed point $(x,y)\in\H^+$, which means that in the sense of distributions on $\H^+$:
$$
L_{m,\xi,\eta}^\star E_m(x,y,\xi,\eta)=\delta_{(x,y)}(\xi,\eta).
$$
Moreover, if $(\xi,\eta)\in\H^+$ is fixed, then in the sense of distributions on $\H^+$ :
$$
L_{m,x,y}E_m(x,y,\xi,\eta)=\delta_{(\xi,\eta)}(x,y),
$$
which means that $E_m$ is a fundamental solution on $\H^+$ of $L_{m,x,y}$ at the fixed point $(\xi,\eta)\in\H^+$.
\end{thm}

\begin{proof}

\def\e{{\varepsilon}}

Let $m\in\C$ and $u\in\mathcal{ D}(\H^+)$. Let $(x,y)\in\H^+$ and $\e>0$ such that $D((x,y),\e)\subset\H^+$where $D((x,y),\e)$ is the disk of center $(x,y)$ and of radius $\e$.

We put
$$
I_\e:=\int_{\H^+\setminus D((x,y),\e)} L_m(u)(\xi,\eta)E_m(x,y,\xi,\eta) d\xi d\eta=
$$
$$
=\int_{\H^+\setminus D((x,y),\e)} \left(L_m(u)(\xi,\eta)E_m(x,y,\xi,\eta)-u(\xi,\eta)L_m^\star(E_m)(x,y,\xi,\eta)\right)d\xi d\eta
$$
because $L_m^\star(E_m)=0$ on $\H^+\setminus D((x,y),\e)$.
An elementary calculation gives us
$$
L_m(u)E_m-uL_m^\star(E)=\d_\xi\left(\!(\d_\xi u)E_m-u(\d_\xi E_m)+\frac m\xi uE_m\!\right)\qquad\qquad\qquad\qquad\qquad
$$
$$
\qquad\qquad\qquad\qquad +\d_\eta\left((\d_\eta u)E_m-u(\d_\eta E_m)\right).
$$
We will recall the Green formula in the framework that will be useful to us here.

\bigskip

{\bf Recall.} {\it 
Let $\Omega$ be an open set of $\R^2$ whose boundary is piecewise ${C}^1$-differentiable. 

By denoting $\vec n$ the outer unit normal vector to $\d\Omega$ and $ds$ the arc length element on $\d\Omega$ (positively oriented), if $X=(X_1,X_2):\overline\Omega\to\C^2$ is a ${C}^1$ vector field, then
$$
\int_\Omega \div X(x,y)dxdy=\int_{\d\Omega} X(x,y)\cdot\vec n(x,y)ds
$$}

\bigskip

With this reminder applied to the open set $\Omega=U\setminus D((x,y),\e)$ where $U$ is a regular open set of $\H^+$ containing the support of $u$, we have
$$
I_\e=-\int_{t\in[0,2\pi]\atop
(\xi,\eta)=(x,y)+\e(\cos t ,\sin t)}\left(\left((\d_\xi u)E_m-u(\d_\xi E_m)+\frac m\xi uE_m\right)\cos t+\right.\qquad\qquad\qquad\qquad
$$
$$
\qquad\quad\qquad\qquad\qquad\qquad\qquad\qquad\qquad\qquad\Bigl.+\left((\d_\eta u)E_m-u(\d_\eta E_m)\right)\sin t\Bigr) \e dt
$$
Proposition \ref{estimation} shows that
$$
\int_{t\in[0,2\pi]\atop
(\xi,\eta)=(x,y)+\e(\cos t,\sin t)}\left[[(\d_\xi u)+\frac m\xi u]\cos t+(\d_\eta u)\sin t\right]E_m\,\e dt\mathop{\longrightarrow}_{\e\to0+}0
$$
because $\lim_{\e\to0}\e\ln\e=0$.
Then, if we want to prove that $\lim_{\e\to0} I_\e$ exists, we have to prove the existence of
$$
\lim_{\e\to0}\int_{t\in[0,2\pi]\atop
(\xi,\eta)=(x,y)+\e(\cos t,\sin t)}u\left((\d_\xi E_m)\cos t+(\d_\eta E_m)\sin t\right)\e\,dt,
$$
and this limit will be equal to the limit of $I_\e$.

Now, we assume that Re $m\geq1$.

We denote $J_\e$ the integral in the previous expression.
A computation gives
$$
J_\e=-\underbrace{\frac m{2\pi}\int_{t\in[0,2\pi]\atop
(\xi,\eta)=(x,y)+\e(\cos t,\sin t)}u\frac{\xi^{m-1}}{\e^m}\int_0^\pi\frac{\sin^{m-1}\theta\,d\theta}{\left(1+k\sin^2\frac\theta2\right)^{m/2}}\e\cos t\,dt}_{J_{\e,1}}+
$$
$$
+\underbrace{\frac m{2\pi}\int_{t\in[0,2\pi]\atop
(\xi,\eta)=(x,y)+\e(\cos t,\sin t)}u\frac{\xi^{m}}{\e^{m+2}}\int_0^\pi\frac{\sin^{m-1}\theta\,d\theta}{\left(1+k\sin^2\frac\theta2\right)^{m/2+1}}\e^2\,dt}_{J_{\e,2}}+
$$
$$
+\underbrace{\frac m{2\pi}\int_{t\in[0,2\pi]\atop
(\xi,\eta)=(x,y)+\e(\cos t,\sin t)}u\frac{\xi^{m}}{\e^{m+2}}\int_0^\pi\frac{2x\sin^2\frac\theta2\sin^{m-1}\theta\,d\theta}{\left(1+k\sin^2\frac\theta2\right)^{m/2+1}}\e\cos t\,dt}_{J_{\e,3}}
$$
where $k=\frac{4x\xi}{\e^2}$.

We have the following propositions : 

\begin{prop}\label{Equivalent2}
When $k\to+\infty$ and $m\in\C$
$$
\int_{\theta=0}^\pi\frac{\sin^2{\theta\over2}\sin^{m-1}\theta\,d\theta}{(1+k\sin^2\frac\theta2)^{m/2+1}}\mathop\sim_{k\to+\infty}
\frac{2^{m-1}}{k^{\frac m2+1}}\ln k.
$$
\end{prop}

\begin{proof}

We put $u=\sin\frac\theta2$, this integral is equal to
$$
2^m\int_0^1\frac{u^{m+1}(1-u^2)^{\frac{m-2}2}du}{(1+ku^2)^{m/2+1}}={2^m\over k^{m/2+1}}\int_0^1\frac{u^{m+1}(1-u^2)^{\frac{m-2}2}du}{(\frac1k+u^2)^{m/2+1}}.
$$
However
$$
\int_0^1\frac{u^{m+1}(1-u^2)^{\frac{m-2}2}du}{(\frac1k+u^2)^{m/2+1}}-\int_0^1\frac{u^{m+1}du}{(\frac1k+u^2)^{m/2+1}}=
-\int_0^1\frac{u^{m+1}}{\left(\frac1k+u^2\right)^{m/2+1}}(1-(1-u^2)^{\frac{m-2}2})du
$$
$$
\mathop{\longrightarrow}_{k\to+\infty}-\int_0^1\frac{u^{m+1}}{\left(u^2\right)^{m/2+1}}(1-(1-u^2)^{\frac{m-2}2})du=-\int_0^1\frac{1-(1-u^2)^{\frac{m-2}2}}{u}du.
$$
The change of variable $u=\frac1{\sqrt k}\sh t$ gives us
$$
\int_0^1\frac{u^{m+1}du}{(\frac1k+u^2)^{m/2+1}}=\int_0^{\argsh\sqrt k}\th^{m+1} t\,dt
$$
Since $\th^{m+1}t$ tends to 1 when $t\to+\infty$, it follows that when $k\to+\infty$
$$
\int_0^{\argsh\sqrt k}\th^{m+1}dt\mathop{\sim}_{k\to+\infty}\int_0^{\argsh\sqrt k}dt=\argsh\sqrt k\mathop{\sim}_{k\to+\infty}\frac12\ln k.
$$
The proposition is well proven..

\end{proof}

\begin{prop}\label{Equivalent3}
When $k\to+\infty$ and $m\in\C$
$$
\int_{\theta=0}^\pi\frac{\sin^{m-1}\theta\,d\theta}{(1+k\sin^2\frac\theta2)^{m/2+1}}\mathop\sim_{k\to+\infty}
\frac{2^m}{mk^{\frac m2}}
$$
\end{prop}

\begin{proof}
Putting as previously $u=\sin\frac\theta2$, this integral is equal to
$$
2^m\int_0^1\frac{u^{m-1}(1-u^2)^{\frac{m-2}2}du}{(1+ku^2)^{m/2+1}}={2^m\over k^{m/2+1}}\int_0^1\frac{u^{m-1}(1-u^2)^{\frac{m-2}2}du}{(\frac1k+u^2)^{m/2+1}}.
$$
However
$$
\int_0^1\frac{u^{m-1}(1-u^2)^{\frac{m-2}2}du}{(\frac1k+u^2)^{m/2+1}}-\int_0^1\frac{u^{m-1}du}{(\frac1k+u^2)^{m/2+1}}=
-\int_0^1\frac{u^{m-1}}{\left(\frac1k+u^2\right)^{m/2+1}}(1-(1-u^2)^{\frac{m-2}2})du
$$

We first estimate the right hand side of this equality :

$$
\int_0^1\frac{u^{m-1}}{\left(\frac1k+u^2\right)^{m/2+1}}(1-(1-u^2)^{\frac{m-2}2})du-\int_0^1\frac{u^{m-1}}{\left(\frac1k+u^2\right)^{m/2+1}}\left(\frac{m-2}{2}u^2\right)du
$$
$$
=\int_0^1\frac{u^{m-1}}{\left(\frac1k+u^2\right)^{m/2+1}}\left(1-\frac{m-2}{2}u^2-(1-u^2)^{\frac{m-2}2}\right)du
$$
$$
\mathop{\longrightarrow}_{k\to+\infty} \int_0^1\frac{u^{m-1}}{\left(u^2\right)^{m/2+1}}\left(1-\frac{m-2}{2}u^2-(1-u^2)^{\frac{m-2}2}\right)du
$$
$$
\qquad\qquad=\int_0^1\frac{1-\frac{m-2}{2}u^2-(1-u^2)^{\frac{m-2}2}}{u^3}du. \eqno{(*)}
$$
As seen in the proof of Proposition \ref{Equivalent2}, we have
$$
{m-2\over 2}\int_{0}^1 {u^{m+1}\over ({1\over k}+u^2)^{{m\over 2}+1}}du\mathop\sim_{k\to+\infty}{m-2\over 4}\ln k.   \eqno{(**)}
$$
Through (*) and (**), one obtains :
$$
\int_0^1\frac{u^{m-1}}{\left(\frac1k+u^2\right)^{m/2+1}}(1-(1-u^2)^{\frac{m-2}2})du\mathop\sim_{k\to+\infty}{m-2\over 4}\ln k.
$$

The change of variable $u=\frac1{\sqrt k}\sh t$ gives us
$$
\int_0^1\frac{u^{m-1}du}{(\frac1k+u^2)^{m/2+1}}=k\int_0^{\argsh\sqrt k}{\th^{m-1} t\over\ch^2 t}\,dt={k\over m}\th^m\left(\argsh\sqrt k\right).
$$
It follows that when $k\to+\infty$,
$$
\int_0^1\frac{u^{m-1}du}{(\frac1k+u^2)^{m/2+1}}\mathop\sim_{k\to+\infty} {k\over m}.
$$

We thus obtain
$$
\int_0^1\frac{u^{m-1}(1-u^2)^{\frac{m-2}2}du}{(\frac1k+u^2)^{m/2+1}}\mathop\sim_{k\to+\infty} {k\over m}.
$$
And
$$
\int_{\theta=0}^\pi\frac{\sin^{m-1}\theta\,d\theta}{(1+k\sin^2\frac\theta2)^{m/2+1}}\mathop\sim_{k\to+\infty}
\frac{2^m}{mk^{\frac m2}}
$$
and this completes the proof.
\end{proof}

Let us return to the proof of Theorem \ref{theo}.

The Proposition \ref{Equivalent1}  shows that
$$
J_{\e,1}\mathop{\sim}_{\e\to0+}-\frac m{2\pi}\int_{t\in[0,2\pi]\atop
(\xi,\eta)=(x,y)+\e(\cos t,\sin t)}u\frac{x^{m-1}}{\e^m}\frac{2^{m-1}}{k^{m/2}}(\ln k)\,\e\cos t\,dt
$$
$$
\mathop{\sim}_{\e\to0+}+\frac m{2\pi x}\e\ln \e\left(\int_{t\in[0,2\pi]\atop
(\xi,\eta)=(x,y)+\e(\cos t,\sin t)}u(x+\e\cos t,y+\e\sin t)\cos t\,dt\right)
$$ 
which tends to 0.

The Proposition \ref{Equivalent2}  shows that
$$
J_{\e,3}\mathop{\sim}_{\e\to0+}\frac m{2\pi}\int_{t\in[0,2\pi]\atop
(\xi,\eta)=(x,y)+\e(\cos t,\sin t)}u\frac{x^m}{\e^{m+2}}(2x)\frac{2^{m-1}}{k^{m/2+1}}(\ln k)\,\e\cos t\,dt
$$
$$
\mathop{\sim}_{\e\to0+}-\frac m{4\pi x}\e\ln \e\left(\int_{t\in[0,2\pi]\atop
(\xi,\eta)=(x,y)+\e(\cos t,\sin t)}u(x+\e\cos t,y+\e\sin t)\cos t\,dt\right)
$$
which tends to 0.

Finally, the proposition \ref{Equivalent3} shows that
$$
J_{\e,2}\mathop{\sim}_{\e\to0+}\frac m{2\pi}\int_{t\in[0,2\pi]\atop
(\xi,\eta)=(x,y)+\e(\cos t,\sin t)}u\frac{x^m}{\e^{m+2}}\frac{2^m}{mk^{m/2}}\e^2\,dt
$$
$$
\mathop{\sim}_{\e\to0+}\frac 1{2\pi}\int_{t\in[0,2\pi]\atop
(\xi,\eta)=(x,y)+\e(\cos t,\sin t)}u(x+\e\cos t,y+\e\sin t)dt\mathop{\longrightarrow}_{\e\to0+}u(x,y).
$$

%

So we have proved that for all $m\in\C$ such that Re $m>0$,
$$
\lim_{\e\to0+}\int_{\H^+\setminus D((x,y),\e)} L_m(u)(\xi,\eta)E_m(x,y,\xi,\eta) d\xi d\eta=\qquad\qquad\qquad\qquad
$$
$$
\qquad\qquad\qquad\qquad=
\int_{\H^+} L_m(u)(\xi,\eta)E_m(x,y,\xi,\eta) d\xi d\eta=u(x,y)
$$
therefore that $E_m$ is indeed a  fundamental solution of $L_m^\star$ for all $m\in\C$ with Re $m>0$.

Proof for $m\in\C$ with Re $m\leq1$ is completely similar.

We also have the dual assertions for fundamental solutions of $L_m$ for all $m\in\C$ thanks to Proposition \ref{conjugaison}.

\end{proof}

The following proposition is roughly a consequence of the previous theorem.
\begin{prop} \label{RepresentationIntegrale}
Let $m\in\C$ and let $\Omega$ be a relatively compact open set of $\H^+$ whose boundary is piecewise $C^1$-differentiable.

Then, for $(x,y)\in\Omega$ and $u\in{C}^2(\overline\Omega)$, by denoting $\vec n$ the outer unit normal vector to $\d\Omega$ and $ds$ the arc length element on $\d\Omega$ (positively oriented), we have
$$
u(x,y)=\int_\Omega L_m(u)E_m\,d\xi d\eta\qquad\qquad\qquad\qquad\qquad\qquad\qquad\qquad\qquad\qquad\qquad\qquad\qquad
$$
$$
\qquad\qquad\qquad\qquad-\int_{\d\Omega}\left[(\d_\xi u)E_m-u(\d_\xi E_m)+\frac m\xi uE_m\,,\,(\d_\eta u)E_m-u(\d_\eta E_m)\right]\cdot \,\vec n\,ds
$$
where $u:=u(\xi,\eta)$ and $E_m:=E_m(x,y,\xi,\eta)$.
\end{prop}

\begin{proof}
Indeed, if $u\in{C}^2(\overline\Omega)$, we have for $(x,y)\in\Omega$ and $\e>0$ such that $\overline{D((x,y),\e)}\subset\Omega$ :
$$
\int_{\Omega\setminus D((x,y),\e)} L_m(u)E_m d\xi\,d\eta=\int_{\Omega\setminus D((x,y),\e)} (L_m(u)E_m -L_m^\star(E_m)u)d\xi\,d\eta.
$$
Thanks to the Green formula previously recalled, this last integral is equal to
$$
\int_{\d\Omega}\left[(\d_\xi u)E_m-u(\d_\xi E_m)+\frac m\xi uE_m\,,\,(\d_\eta u)E_m-u(\d_\eta E_m)\right]\cdot \,\vec n\,ds
$$
$$
-\int_{t\in[0,2\pi]\atop
(\xi,\eta)=(x,y)+\e(\cos t ,\sin t)}\left(\left((\d_\xi u)E_m-u(\d_\xi E_m)+\frac m\xi uE_m\right)\cos t+\right.\qquad\qquad\qquad\qquad
$$
$$
\qquad\quad\qquad\qquad\qquad\qquad\qquad\qquad\qquad\qquad\qquad\Bigl.+\left((\d_\eta u)E_m-u(\d_\eta E_m)\right)\sin t\Bigr) \e dt,
$$
and from what we saw in the previous proof, this last expression tends to
$$
\int_{\d\Omega}\left[(\d_\xi u)E_m-u(\d_\xi E_m)+\frac m\xi uE_m\,,\,(\d_\eta u)E_m-u(\d_\eta E_m)\right]\cdot \,\vec n\,ds+u(x,y),
$$
when $\e\to0$.

The integrability nature of $E_m$ near $(x,y)$ shows that
$$
\lim_{\e\to0}\int_{\Omega\setminus D((x,y),\e)} L_m(u)E_m d\xi\,d\eta=\int_{\Omega} L_m(u)E_m d\xi\,d\eta,
$$
and the proof of the proposition is complete.
\end{proof}

\section{Liouville-type result and decomposition theorem for the axisymmetric potentials}

In the previous section, we have just seen that there are two different expressions of the fundamental solutions depending on the values ​​of $m$. For the rest, each of the expressions have different behaviors according to the value of $m$. We will look at the two cases separately : Re $m<1$ and Re $m\geq1$.

More specifically, we need fundamental solutions which vanish at the boundary of $\H^+$, 
that is to say zero on the $y$-axis and zero to infinity.

For Re $m<1$, the formula
$$
E_m(x,y,\xi,\eta)=-{\xi x^{1-m}\over 2\pi}\int_{\theta=0}^{\pi}{\sin^{1-m}\theta \, d\theta \over \left[(x-\xi)^2+4x\xi\sin^2\left({\theta\over 2}\right)+(y-\eta)^2\right]^{1-{m\over2}}}
$$
shows that $E_m$ satisfies this property ($E_m(x,y,\cdot,\cdot)$ tends to 0 when $x\to0+$ and $\|(x,y)\|\to+\infty$).

For Re $m\geq1$, the expression
$$
E_m(x,y,\xi,\eta)=-\frac{\xi^m}{2\pi}\int_{\theta=0}^{\pi}\frac{\sin^{m-1}\theta\,d\theta}{[(x-\xi)^2+4x\xi\sin^2\left(\frac\theta2\right)+(y-\eta)^2]^{m/2}}
$$
no longer satisfies this property. Contrariwise,
$$
E_m(x,y,\xi,\eta)-E_m(-x,y,\xi,\eta)
$$
is also a fundamental solution on $\H^+$, and satisfies this property.

Then, we will put

- For Re $m<1$ :
$$
F_m(x,y,\xi,\eta)=E_m(x,y,\xi,\eta)
$$
- For Re $m\ge1$ :
$$
F_m(x,y,\xi,\eta)=E_m(x,y,\xi,\eta)-E_m(-x,y,\xi,\eta).
$$

We will need the following definition of convergence to the boundary of $\H^+$.

\bigskip\noindent
{\bf Definition.} Let $u:\H^+\to\R$ be a function defined on $\H^+$. We write
$$
\lim_{\d\H^+}u=0
$$
if and only if
$$
\forall \e>0,\qquad \exists N\in\N,\qquad \forall n\geq N,\qquad\forall (x,y)\in\H^+,\qquad\qquad\qquad\qquad\qquad\qquad\qquad$$
$$
 x\leq{1\over n}\hbox{ \rm or } \|(x,y)\|\geq n\quad\Longrightarrow\quad |u(x,y)|\leq \e.
$$
\bigskip
\noindent

In other words, this amounts to considering that the boundary $\d\H^+$ of $\H^+$ consists of $y$-axis points and points at infinity and to say that the concept of punctual convergence to the boundary of $\H^+$ is a uniform convergence. Indeed, we do not need the uniform convergence. More precisely, we have the following proposition :

\begin{prop}
Let $u:\H^+\to\C$. We have
$$
\lim_{\d\H^+}u=0
$$
if and only if
$$
\lim_{\|(x,y)\|\to+\infty}u(x,y)=0\qquad\hbox{and}\qquad\forall y\in\R,\quad \lim_{(0,y)}u=0.
$$
\end{prop}
\begin{proof}
The direct implication is easy. Conversely, we assume
$$
\lim_{\|(x,y)\|\to+\infty}u(x,y)=0\qquad\hbox{and}\qquad\forall y\in\R,\quad \lim_{(0,y)}u=0
$$
and we have to show $\lim_{\d\H^+}u=0$.

Let $\e>0$. There is $A>0$ such that for all $(\xi,\eta)\in\H^+$, 
$$
\sqrt{\xi^2+\eta^2}\geq A\qquad\Rightarrow\qquad |u(\xi,\eta)|\leq\e.
$$
Similarly, for all $y\in\R$, there is $\alpha_y\in]0,1[$ such that for all $(\xi,\eta)\in\H^+$
$$
\sqrt{\xi^2+(\eta-y)^2}<\alpha_y\qquad\Rightarrow |u(\xi,\eta)|\leq\e.
$$
The interval $[-A,A]$ is compact. 

By the Lebesgue covering lemma, there is $\alpha>0$ such that for all $y'\in[-A,A]$, the ball $B(y',\alpha)$ is included in one of the balls $B(y,\alpha_{y})$ with $y\in[-A,A]$. 

In particular, if $(\xi,\eta)\in\H^+$ is such that $0<\xi<\alpha$, then $|u(\xi,\eta)|\leq\e$.
This completes the proof.
\end{proof}

\bigskip
The following proposition is a Liouville-type result for the axisymmetric potentials in the right half-plane and this result is not immediate because there is the loss of strict ellipticity of the Weinstein operator on the $y$-axis. In \cite{BBC} (see Theorem 7.1), we can found an interesting result on the description of a class of non-strictly elliptic equations with unbounded coefficients.

\begin{prop}\label{Unicite}
Let $u\in C^2(\H^+)$ such that $L_mu=0$ and $\displaystyle \lim_{\d\H^+}u=0$.
Then $u\equiv 0$ on $\H^+$.
 \end{prop}

\begin{proof}

For $(\xi,\eta)\in\H^+$ and $N\in\N^*$, we define
$$
\phi_N(\xi,\eta)=\theta_1(N\xi)\theta_2\left({\xi\over N}\right)\theta_2\left({\eta\over N}\right)
$$
where $\theta_1$ and $\theta_2$ are smooth functions on $\R$, valued on $[0,1]$ and such that $\theta_1(t)=1$ for $t\geq 1$, $\theta_1(t)= 0$ for $t\leq {1\over2}$, $\theta_2(t)=1$ for $t\in\left[-{1\over2},{1\over 2}\right]$ and $\theta_2(t)=0$ for $t\in\R\setminus\left]-1,1\right[$. We assume that all derivatives of $\theta_1$ and $\theta_2$ vanish at $\left\{-1,-{1\over2},{1\over 2},1\right\}$.

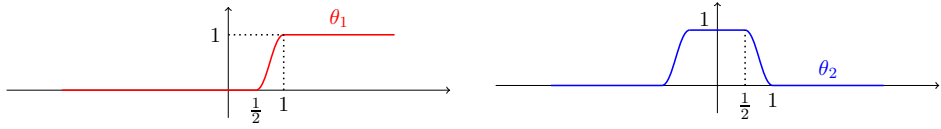
\begin{multicols}{2}
\resizebox{6cm}{!}{
\begin{tikzpicture}
\draw[->] (-4,0) -- (4,0);
\draw[->] (0,-0.5) -- (0,1.5);
\draw[thick,red,-] (0.5,0) .. controls (0.7,0) and (0.8,1) .. (1,1);
\draw [thick,red,-] (1,1)--(3,1);
\draw [thick,red,-] (-3,0)--(0.5,0);
\draw[thick,dotted,-] (0,1)--(1,1);
\draw[thick,dotted,-] (1,0)--(1,1);
\draw (0,1) node[left] {$1$};
\draw (1,0) node[below] {$1$};
\draw (0.5,0) node[below] {${1\over2}$};
\draw[red] (2,1) node[above] {$\theta_1$};
\end{tikzpicture}}

\resizebox{6cm}{!}{
\begin{tikzpicture}
\draw[->] (-4,0) -- (4,0);
\draw[->] (0,-0.5) -- (0,1.5);
\draw[thick,blue,-] (-1,0) .. controls (-0.8,0) and (-0.7,1) .. (-0.5,1);
\draw[thick,blue,-] (1,0) .. controls (0.8,0) and (0.7,1) .. (0.5,1);
\draw [thick,blue,-] (1,0)--(3,0);
\draw [thick,blue,-] (-3,0)--(-1,0);
\draw [thick,blue,-] (-0.5,1)--(0.5,1);
\draw[thick,dotted,-] (0.5,0)--(0.5,1);
\draw (0,1.2) node[left]{$1$};
\draw (1,0) node[below] {$1$};
\draw (0.5,0) node[below] {${1\over2}$};
\draw[blue] (2,0) node[above] {$\theta_2$};
\end{tikzpicture}}
\end{multicols}

If $u\in{C}^2(\H^+)$ satisfies $L_mu=0$, then $u\phi_N\in{C}^2(\H^+)$ and is compactly supported on $\H^+$. Throughout the following, we fix $(x,y)\in\H^+$. For $N$ sufficiently large, thanks to Proposition \ref{RepresentationIntegrale} (true if $E_m$ is replaced by $F_m$), we have
$$
u(x,y)=u(x,y)\phi_N(x,y)=\int_{\H^+} L_m(u\phi_N)F_m\,d\xi d\eta
$$
(because the function $L_m(u\phi_N)$  is identically zero in a neighborhood of the singularity of $F_m$), thus
$$
u(x,y)
=\int_{\H^+} [L_m(u)\phi_N+uL_m(\phi_N)+2\nabla u\cdot\nabla\phi_N]F_m\,d\xi d\eta
$$
$$
=\int_{\H^+}u [L_m(\phi_N)F_m-2 \div\left(F_m\nabla\phi_N\right)]d\xi d\eta
$$
$$
=\int_{D_1\cup\cdots\cup D_8}u [L_m(\phi_N)F_m-2 \div\left(F_m\nabla\phi_N\right)]d\xi d\eta
$$
$$
=-\int_{D_1\cup\cdots\cup D_8}u [L_{-m}(\phi_N)F_m+2 \nabla F_m\cdot\nabla\phi_N]d\xi d\eta
$$
where $D_1,\ldots, D_8$ are the following domains (which depend of  N) :
$$
D_1=\left[{1\over 2N},{1\over N}\right] \times \left[-{N\over 2},{N\over 2}\right],
\quad
D_2=\left[{1\over N},{N\over 2}\right] \times\left[{N\over 2},N\right],
$$

$$
D_3=\left[{N\over 2},N\right] \times \left[-{N\over 2},{N\over 2}\right],
\quad
D_4=\left[{1\over N},{N\over 2}\right] \times \left[-N,-{N\over 2},\right],
$$

$$
D_5=\left[{1\over 2N},{1\over N}\right] \times\left[{N\over 2},N\right],
\quad
D_6=\left[{N\over 2},N\right] \times\left[{N\over 2},N\right],
$$

$$
D_7=\left[{N\over 2},N\right] \times\left[-N,-{N\over 2}\right]
\quad{\rm and}\quad
D_8=\left[{1\over 2N},{1\over N}\right] \times\left[-N,-{N\over 2}\right].
$$

\hfil
\resizebox{6cm}{!}{
\begin{tikzpicture}

\draw (0.5,-4.8) node{Figure  : Domains $D_i$};

\draw[-,thick,fill=blue!15!white,opacity=0.8] (-2.5,-3)--(-2.5,3)--(3,3)--(3,-3)--cycle;
\draw[-,thick,fill=red!15!white,opacity=0.8] (-2.5,3)--(-2,3)--(-2,2)--(-2.5,2)--cycle;

\draw[-,thick,fill=red!15!white,opacity=0.8] (-2.5,-3)--(-2,-3)--(-2,-2)--(-2.5,-2)--cycle;

\draw[-,thick,fill=red!15!white,opacity=0.8] (2,3)--(3,3)--(3,2)--(2,2)--cycle;

\draw[-,thick,fill=red!15!white,opacity=0.8] (2,-3)--(3,-3)--(3,-2)--(2,-2)--cycle;

\draw[-,thick,fill=white] (-2,-2)--(-2,2)--(2,2)--(2,-2)--cycle;

\draw[thick] (0,2.5) node {$D_2$};
\draw[thick] (0,-2.5) node {$D_4$};
\draw[thick] (2.5,2.5) node {$D_6$};
\draw[thick] (2.5,-2.5) node {$D_7$};
\draw[thick] (2.5,0.5) node {$D_3$};
\draw[thick] (-2.25,0.5) node {$D_1$};
\draw[thick] (-2.25,2.5) node {$D_5$};
\draw[thick] (-2.25,-2.5) node {$D_8$};

\draw[<->,thick] (-3,1)--(-2.5,1);
\draw[thick] (-2.8,0.7) node {\, \scriptsize$ {1\over 2N}$\,\,};

\draw[->] (-4.,0) -- (4,0);
\draw[->] (-3,-4.5) -- (-3,4);

\draw[<->,thick] (-3,-1)--(-2,-1);
\draw[thick] (-2.8,-1.3) node {\, \scriptsize$ {1\over N}$\,\,};

\draw[<->,thick] (-3,-4)--(3,-4);
\draw[thick] (0,-4.2) node {\scriptsize $N$};
\draw[-,thick,dotted] (3,-4)--(3,-3);

\draw[<->,thick] (-3,-3.5)--(2,-3.5);
\draw[thick] (0,-3.7) node {\scriptsize $N/2$};
\draw[-,thick,dotted] (2,-3.5)--(2,-3);

\draw[thick] (0,0.7) node {$\phi_N\equiv 1$};

\draw[-,thick,dotted] (-2.5,-3)--(-3,-3);
\draw[thick] (-3,-3) node[left] {\, \scriptsize$- N$\,\,};

\draw[-,thick,dotted] (-2.5,-2)--(-3,-2);
\draw[thick] (-3,-2) node[left] {\, \scriptsize$- N/2$\,\,};

\draw[-,thick,dotted] (-2.5,3)--(-3,3);
\draw[thick] (-3,3) node[left] {\, \scriptsize$N$\,\,};

\draw[-,thick,dotted] (-2.5,2)--(-3,2);
\draw[thick] (-3,2) node[left] {\, \scriptsize$ N/2$\,\,};

\draw (3.8,0.2) node {$\xi$};
\draw (-2.7,3.8) node {$\eta$};

\end{tikzpicture}}
\hfill

Since $\lim_{\d\H^+}u=0$, then
$$
u_N:=\sup_{(\xi,\eta)\in D_1\cup\cdots\cup D_8}|u(\xi,\eta)|\mathop{\longrightarrow}_{N\to+\infty}0.
$$
We will estimate each integrals supported on $D_1,\ldots,D_8$. For this, we need the following lemmas which will give us estimates of each terms when $N$ tends to infinity. We recall that, if $(u_N)_N$ and $(v_N)_N$ are complex sequences, $u_N=\O(v_N)$ means that there exists a constant $M$ such that, for every $N$ sufficiently large, $|u_N|\leq M|v_N|$ ; $u_N=o(v_N)$ means that for every  $\e>0$, for every $N$ sufficiently large, $|u_N|\leq \e|v_N|$.
\bigskip

\begin{lem} On $D_1$, we have
$$
\sup\left|{\d\phi_N\over\d\xi}\right|=\O(N)\qquad\hbox{and}\qquad\sup\left|{\d\phi_N\over\d\eta}\right|=0.
$$
On $D_2\cup D_4$, we have
$$
\sup\left|{\d\phi_N\over\d\xi}\right|=0\qquad\hbox{and}\qquad\sup\left|{\d\phi_N\over\d\eta}\right|=\O\left({1\over N}\right).
$$
On $D_3$, we have
$$
\sup\left|{\d\phi_N\over\d\xi}\right|=\O\left({1\over N}\right)\qquad\hbox{and}\qquad\sup\left|{\d\phi_N\over\d\eta}\right|=0.
$$
On $D_5\cup D_8$, we have
$$
\sup\left|{\d\phi_N\over\d\xi}\right|=\O(N)\qquad\hbox{and}\qquad\sup\left|{\d\phi_N\over\d\eta}\right|=\O\left({1\over N}\right).
$$
On $D_6\cup D_7$, we have
$$
\sup\left|{\d\phi_N\over\d\xi}\right|=\O\left({1\over N}\right)\qquad\hbox{and}\qquad\sup\left|{\d\phi_N\over\d\eta}\right|=\O\left({1\over N}\right).
$$
On $D_1\cup D_5\cup D_8$, we have
$$
\sup\left|L_{-m}(\phi_N)\right|=\O(N^2).
$$
On $D_2\cup D_3\cup D_4\cup D_6\cup D_7$, we have
$$
\sup\left|L_{-m}(\phi_N)\right|=\O\left({1\over N^2}\right).
$$
\end{lem}

\begin{proof}

*For $(\xi,\eta)\in D_1$, $\phi_N(\xi,\eta)=\theta_1(N\xi)$ and thus
$$
{\d\phi_N\over\d\xi}(\xi,\eta)=N\theta_1'(N\xi)\qquad\hbox{,}\qquad{\d\phi_N\over\d\eta}(\xi,\eta)=0,
$$
$$
L_{-m}\phi_N(\xi,\eta)=N^2\theta_1''(N\xi)-{m \,N\over \xi}\theta_1'(N\xi),
$$
which give us
$$
\sup_{D_1}\left|{\d\phi_N\over\d\xi}\right|=\O(N),\qquad \sup_{D_1}\left|{\d\phi_N\over\d\eta}\right|=0,\qquad\sup_{D_1}|L_{-m}(\phi_N)|=\O(N^2)
$$
since the derivatives of $\theta_1$ are bounded and for $(\xi,\eta)\in D_1$, one gets $\xi\geq{1\over 2N}$.

\bigskip
*For $(\xi,\eta)\in D_2$, $\phi_N(\xi,\eta)=\theta_2\left({\eta\over N}\right)$ and thus
$$
{\d\phi_N\over\d\xi}(\xi,\eta)=0\qquad\hbox{,}\qquad{\d\phi_N\over\d\eta}(\xi,\eta)={1\over N}\theta_2'\left({\eta\over N}\right),
$$
$$
L_{-m}\phi_N(\xi,\eta)={1\over N^2}\theta_2''\left({\eta\over N}\right),
$$
which give us
$$
\sup_{D_2}\left|{\d\phi_N\over\d\xi}\right|=0,\qquad \sup_{D_2}\left|{\d\phi_N\over\d\eta}\right|=\O\left({1\over N}\right),\qquad\sup_{D_2}|L_{-m}(\phi_N)|=\O\left({1\over N^2}\right)
$$

*So does same with $D_4$.

\bigskip
* For $(\xi,\eta)\in D_3$, $\phi_N(\xi,\eta)=\theta_2\left({\xi\over N}\right)$ and thus
$$
{\d\phi_N\over\d\xi}(\xi,\eta)={1\over N}\theta_2'\left({\xi\over N}\right)\qquad\hbox{,}\qquad{\d\phi_N\over\d\eta}(\xi,\eta)=0,
$$
$$
L_{-m}\phi_N(\xi,\eta)={1\over N^2}\theta_2''\left({\xi\over N}\right)-{1\over N}{m\over\xi}\theta_2'\left({\xi\over N}\right),
$$
which give us
$$
\sup_{D_3}\left|{\d\phi_N\over\d\xi}\right|=\O\left({1\over N}\right),\qquad \sup_{D_3}\left|{\d\phi_N\over\d\eta}\right|=0,\qquad\sup_{D_3}|L_{-m}(\phi_N)|=\O\left({1\over N^2}\right)
$$
\bigskip
* For $(\xi,\eta)\in D_5$, $\phi_N(\xi,\eta)=\theta_1(N\xi)\theta_2\left({\eta\over N}\right)$ and thus
$$
{\d\phi_N\over\d\xi}(\xi,\eta)=N\theta_1'(N\xi)\theta_2\left({\eta\over N}\right)\qquad\hbox{,}\qquad{\d\phi_N\over\d\eta}(\xi,\eta)={1\over N}\theta_1(N\xi)\theta_2'\left({\eta\over N}\right),
$$
$$
L_{-m}\phi_N(\xi,\eta)=N^2\theta_1''(N\xi)\theta_2\left({\eta\over N}\right)+{1\over N^2}\theta_1(N\xi)
\theta_2''\left({\eta\over N}\right)-{m\over\xi}N\theta_1'(N\xi)\theta_2\left({\eta\over N}\right)
$$
which give us
$$
\sup_{D_5}\left|{\d\phi_N\over\d\xi}\right|=\O(N),\qquad \sup_{D_5}\left|{\d\phi_N\over\d\eta}\right|=\O\left({1\over N}\right),\qquad\sup_{D_5}|L_{-m}(\phi_N)|=\O(N^2).
$$

* So does same with  $D_8$.

\bigskip
*For $(\xi,\eta)\in D_6$, $\phi_N(\xi,\eta)=\theta_2\left({\xi\over N}\right)\theta_2\left({\eta\over N}\right)$ and thus
$$
{\d\phi_N\over\d\xi}(\xi,\eta)={1\over N}\theta_2'\left({\xi\over N}\right)\theta_2\left({\eta\over N}\right)\qquad\hbox{,}\qquad{\d\phi_N\over\d\eta}(\xi,\eta)={1\over N}\theta_2\left({\xi\over N}\right)\theta_2'\left({\eta\over N}\right),
$$
$$
L_{-m}\phi_N(\xi,\eta)={1\over N^2}\theta_2''\left({\xi\over N}\right)\theta_2\left({\eta\over N}\right)+{1\over N^2}\theta_2\left({\xi\over N}\right)
\theta_2''\left({\eta\over N}\right)-{m\over N\xi}\theta_2'\left({\xi\over N}\right)\theta_2\left({\eta\over N}\right)
$$
which give us
$$
\sup_{D_6}\left|{\d\phi_N\over\d\xi}\right|=\O\left({1\over N}\right),\qquad \sup_{D_6}\left|{\d\phi_N\over\d\eta}\right|=\O\left({1\over N}\right),\qquad\sup_{D_6}|L_{-m}(\phi_N)|=\O\left({1\over N^2}\right).
$$

* So does same with $D_7$. Hence the lemma resulting.

\end{proof}

We now estimate the following quantities for $i\in\{1,\,\ldots,\,8\}$ :
$$
\int_{D_i} |F_m| \,d\xi d\eta, \quad \int_{D_i} |\d_\xi F_m| \,d\xi d\eta\quad {\rm et}\, \int_{D_i} |\d_\eta F_m| \,d\xi d\eta.
$$
\begin{lem}\label{estimationMleq1}
For Re $ m<1$, we have : 

- for $i=1 :$
$$
\int_{D_i}|F_m|d\xi\,d\eta=\O\left({1\over N^2}\right),\qquad\int_{D_i}\left|{\d F_m\over \d \xi}\right|d\xi d\eta=\O\left({1\over N}\right).
$$
- for $i=2, 4:$
$$
\int_{D_i}|F_m|d\xi\,d\eta=\O\left({N^2}\right),\qquad\int_{D_i}\left|{\d F_m\over \d \eta}\right|d\xi d\eta=\O\left({N}\right).
$$
- for $i=3:$
$$
\int_{D_i}|F_m|d\xi\,d\eta=\O\left({N^2}\right),\qquad\int_{D_i}\left|{\d F_m\over \d \xi}\right|d\xi d\eta=\O\left({N}\right).
$$
- for $i=5,8:$
$$
\int_{D_i}|F_m|d\xi\,d\eta=\O\left({1\over N^2}\right),\qquad\int_{D_i}\left|{\d F_m\over \d \xi}\right|d\xi d\eta=\O\left({1\over N}\right),$$
$$
\int_{D_i}\left|{\d F_m\over \d \eta}\right|d\xi d\eta=\O\left({1\over N^2}\right).
$$
- for $i=6,7:$
$$
\int_{D_i}|F_m|d\xi\,d\eta=\O\left({N^2}\right),\qquad\int_{D_i}\left|{\d F_m\over \d \xi}\right|d\xi d\eta=\O\left({N}\right),$$
$$
\int_{D_i}\left|{\d F_m\over \d \eta}\right|d\xi d\eta=\O\left({N}\right).
$$

\end{lem}

\begin{proof}

For Re $m<1$, we have
$$
F_m(\xi,\eta)=-{\xi x^{1-m}\over 2\pi}\int_{\theta=0}^{\pi}{\sin^{1-m}\theta \, d\theta \over \left[(x-\xi)^2+4x\xi\sin^2\left({\theta\over 2}\right)+(y-\eta)^2\right]^{1-{m\over2}}}.
$$
therefore there is a constant $C_1$ such that for all $(\xi,\eta)\in\H^+$, we have
\begin{equation}\label{Fm}
|F_m(\xi,\eta)|\leq{C_1\xi\over\left[(x-\xi)^2+(\eta-y)^2\right]^{1-{{\rm Re\, m}\over 2}}}.
\end{equation}
Similarly, we have
$$
\frac{\d F_m}{\d\xi}=\frac{F_m}\xi-{\xi x^{1-m}\over 2\pi}(m-2)\int_{\theta=0}^{\pi}{[(\xi-x)+2x\sin^2\frac\theta2]\sin^{1-m}\theta \, d\theta \over \left[(x-\xi)^2+4x\xi\sin^2\left({\theta\over 2}\right)+(y-\eta)^2\right]^{2-{m\over2}}},$$
and as before, as
$$
\forall\theta\in[0,\pi],\quad\left|{[(\xi-x)+2x\sin^2\frac\theta2]\sin^{1-m}\theta \,  \over \left[(x-\xi)^2+4x\xi\sin^2\left({\theta\over 2}\right)+(y-\eta)^2\right]^{2-{m\over2}}}\right|\leq{|(\xi-x)+2x\sin^2{\theta\over2}|\over\left(\left(x-\xi\right)^2+(\eta-y)^2\right)^{2-{{\rm Re\, m}\over2}}}
$$ 
$$
={|\xi-x\cos\theta|\over\left(\left(x-\xi\right)^2+(\eta-y)^2\right)^{2-{{\rm Re\, m}\over2}}}\leq{\xi+x\over\left(\left(x-\xi\right)^2+(\eta-y)^2\right)^{2-{{\rm Re\, m}\over2}}},
$$
there exists a constant $C_2$ such that for all $N$ large enough and for all $(\xi,\eta)\in\H^+$, we have
\begin{equation}\label{Fmxi}
\left|\frac{\d F_m}{\d\xi}\right|\leq C_2\left[{1\over\left[\left(x-\xi\right)^2+(\eta-y)^2\right]^{1-{{\rm Re\, m}\over2}}} +{\xi(x+\xi)\over\left(\left(x-\xi\right)^2+(\eta-y)^2\right)^{2-{{\rm Re\, m}\over2}}}\right].
\end{equation}
Finally, as
$$
{\d F_m\over\d\eta}=(2-m)(\eta-y){\xi x^{1-m}\over2\pi}\int_{\theta=0}^{\pi}{\sin^{1-m}\theta\over \left[(x-\xi)^2+4x\xi\sin^2\left({\theta\over 2}\right)+(y-\eta)^2\right]^{2-{m\over2}}},
$$
there exists a constant $C_3$ such that for all $N$ large enough and for all $(\xi,\eta)\in\H^+$, we have
\begin{equation}\label{Fmeta}
\left|{\d F_m\over\d\eta}\right|\leq{C_3\xi\over|\eta-y|^{3-{\rm Re\,}m}}.
\end{equation}

Using these inequalities, we estimate integrals of these functions on the domains $D_i$.

\bigskip$\underline{\hbox{On }D_1}$ :
Inequality (\ref{Fm}) give us
$$
\int_{D_1}|F_m|d\xi d\eta= \O(1)\int_{\xi={1\over 2N}}^{\xi={1\over N}}\int_{\eta=-{N\over2}}^{\eta={N\over2}}{\xi\,d\xi\,d\eta\over\left[(x-\xi)^2+(\eta-y)^2\right]^{1-{{\rm Re\, m}\over 2}}}
$$
$$
=\O(1/N^2)\int_{\eta=-{N\over2}}^{\eta={N\over2}}{d\eta\over\left[(x-{1\over N})^2+(\eta-y)^2\right]^{1-{{\rm Re\, m}\over 2}}}=\O(1/N^2).
$$
Then, thanks to (\ref{Fmxi}), we have
$$
\int_{D_1}\left|{\d F_m\over \d \xi}\right|d\xi d\eta=\O(1)\int_{\xi={1\over 2N}}^{\xi={1\over N}}\int_{\eta=-{N\over2}}^{\eta={N\over2}}\left[{1\over\left[\left(x-\xi\right)^2+(\eta-y)^2\right]^{1-{\Re m\over2}}} \right.\qquad\qquad\qquad\qquad
$$
$$
\qquad\qquad\qquad\qquad\qquad\qquad\qquad\qquad \left. +{\xi(x+\xi)\over\left(\left(x-\xi\right)^2+(\eta-y)^2\right)^{2-{\Re m\over2}}}\right]d\xi d\eta
$$
$$
=\O(1/N)\int_{\eta=-{N\over2}}^{\eta={N\over2}}{d\eta\over\left[(x-{1\over N})^2+(\eta-y)^2\right]^{1-{{\rm Re\, m}\over 2}}}+\O(1/N^2)=\O(1/N).
$$

\bigskip$\underline{\hbox{On }D_2}$ : due to inequality (\ref{Fm}), we have
$$
\int_{D_2}|F_m|d\xi d\eta=\O(1)\int_{\xi={1\over N}}^{{N\over 2}}\int_{\eta={N\over2}}^N {\xi \,d\xi d\eta\over\left[(x-\xi)^2+(\eta-y)^2\right]^{1-{{\rm Re\, m}\over 2}}}
$$
$$
=\O(1)\int_{\xi={1\over N}}^{{N\over 2}}\int_{\eta={N\over2}}^N {\xi\over\left|\eta-y\right|^{2-{\rm Re\, m}}} \,d\xi d\eta=\O(N^2)\int_{\eta={N\over2}}^N {d\eta\over\left|\eta-y\right|^{2-{\rm Re\, m}}}
$$
$$
=\O(N^2)\left[{1\over(N-y)^{1-{\rm Re\,}m}}-{1\over({N\over2}-y)^{1-{\rm Re\,}m}}\right]=\O(N^{{Re\,}m+1}).
$$

Then, thanks to (\ref{Fmeta}), we have
$$
\int_{D_2}\left|{\d F_m\over \d \eta}\right|d\xi d\eta=\O(1)\int_{\xi={1\over N}}^{{N\over 2}}\int_{\eta={N\over2}}^N{\xi\over|\eta-y|^{3-{\rm Re\,}m}}d\xi\,d\eta
$$
$$
=\O(N^2)\int_{\eta={N\over2}}^N{d\eta\over|\eta-y|^{3-{\rm Re\,}m}}=\O(N^{{\rm Re\,}m})
$$

\bigskip$\underline{\hbox{On }D_3}$ : due to inequality (\ref{Fm}), we have
$$
\int_{D_3}|F_m|d\xi d\eta=\O(1)\int_{\xi={N\over2}}^{{N}}\int_{\eta=-{N\over2}}^{N\over2} {\xi \,d\xi d\eta\over\left[(x-\xi)^2+(\eta-y)^2\right]^{1-{{\rm Re\, m}\over 2}}}
$$
$$
=\O(1)\int_{\xi={N\over2}}^{{N}}\int_{\eta=-{N\over2}}^{N\over2} {\xi \,d\xi d\eta\over\left[(x-{N\over2})^2+(\eta-y)^2\right]^{1-{{\rm Re\, m}\over 2}}}$$
$$
=\O(N^2)\int_{\eta=-{N\over2}}^{N\over2} {\,d\eta\over\left[(x-{N\over2})^2+(\eta-y)^2\right]^{1-{{\rm Re\, m}\over 2}}}
$$
$$
=\O(N^2)\int_{\eta=-{N\over2}}^{N\over2} {\,d\eta\over\left[1+(\eta-y)^2\right]^{1-{{\rm Re\, m}\over 2}}}=\O(N^2).
$$
Then, thanks to (\ref{Fmxi}), we have
$$
\int_{D_3}\left|{\d F_m\over \d \xi}\right|d\xi d\eta=\O(1)\int_{\xi={N\over2}}^{{N}}\int_{\eta=-{N\over2}}^{N\over2} \left[{1\over\left[\left(x-\xi\right)^2+(\eta-y)^2\right]^{1-{\Re m\over2}}}\right.\qquad\qquad\qquad\qquad
$$
$$
\qquad\qquad\qquad\qquad\qquad\qquad\qquad\qquad \left.+{\xi(x+\xi)\over\left(\left(x-\xi\right)^2+(\eta-y)^2\right)^{2-{\Re m\over2}}}\right]d\xi d\eta
$$
$$
=\O(N)\int_{\eta=-{N\over2}}^{N\over2} {\,d\eta\over\left[(x-{N\over2})^2+(\eta-y)^2\right]^{1-{{\rm Re\, m}\over 2}}}+\O(N^3)\int_{\eta=-{N\over2}}^{N\over2} {\,d\eta\over\left[(x-{N\over2})^2+(\eta-y)^2\right]^{2-{{\rm Re\, m}\over 2}}}
$$
$$
=\O(N)+\O(N^3)\int_{\eta=-{N\over2}}^{N\over2} {\,d\eta\over (x-{N\over2})^{4-{\rm Re\, m}}}
$$
$$
=\O(N)+\O(N^{{\rm Re\,}m})=\O(N).
$$

\bigskip$\underline{\hbox{On }D_4}$ : this case is analogous to the case $D_2$.

\bigskip$\underline{\hbox{Sur }D_5}$ : due to inequality (\ref{Fm}), we have
$$
\int_{D_5}|F_m|d\xi d\eta=\O(1)\int_{\xi={1\over 2N}}^{{1\over N}}\int_{\eta={N\over2}}^N {\xi \,d\xi d\eta\over\left[(x-\xi)^2+(\eta-y)^2\right]^{1-{{\rm Re\, m}\over 2}}}
$$
$$
=\O(1/N^2)\int_{\eta={N\over2}}^N { d\eta\over (\eta-y)^{2-{\rm Re\, m}}}
$$
$$
=\O(1/N^2)\left[{1\over(N-y)^{1-{\rm Re\,}m}}-{1\over({N\over2}-y)^{1-{\rm Re\,}m}}\right]=\O(1/N^{3-{\rm Re\,}m}).
$$
Then, thanks to (\ref{Fmxi}), we have
$$
\int_{D_5}\left|{\d F_m\over \d \xi}\right|d\xi d\eta=\O(1)\int_{\xi={1\over 2N}}^{\xi={1\over N}}\int_{\eta={N\over2}}^{\eta={N}}\left[{1\over\left[\left(x-\xi\right)^2+(\eta-y)^2\right]^{1-{\Re m\over2}}} \right.\qquad\qquad\qquad\qquad
$$
$$
\qquad\qquad\qquad\qquad\qquad\qquad\qquad\qquad \left. +{\xi(x+\xi)\over\left(\left(x-\xi\right)^2+(\eta-y)^2\right)^{2-{\Re m\over2}}}\right]d\xi d\eta
$$
$$
=\O(1)\int_{\xi={1\over 2N}}^{\xi={1\over N}}\int_{\eta={N\over2}}^{\eta={N}}\left[{1\over\left[\left(x-{1\over N}\right)^2+(\eta-y)^2\right]^{1-{\Re m\over2}}} \right.\qquad\qquad\qquad\qquad
$$
$$
\qquad\qquad\qquad\qquad\qquad\qquad\qquad\qquad \left. +{\xi(x+\xi)\over\left(\left(x-{1\over N}\right)^2+(\eta-y)^2\right)^{2-{\Re m\over2}}}\right]d\xi d\eta
$$
$$
=\O\left({1\over N}\right)
$$
The estimate (\ref{Fmeta}) gives
$$
\int_{D_5}\left|{\d F_m\over \d \eta}\right|d\xi d\eta=\O(1)\int_{\xi={1\over 2N}}^{\xi={1\over N}}\int_{\eta={N\over2}}^{\eta={N}}
{\xi\,d\xi\,d\eta \over|\eta-y|^{3-{\rm Re\,}m}}=\O\left({1\over N^2}\right)
$$

\bigskip$\underline{\hbox{On }D_6}$ : due to (\ref{Fm}), we have
$$
\int_{D_6}|F_m|d\xi d\eta=\O(1)\int_{\xi={N\over2}}^{{N}}\int_{\eta={N\over2}}^N {\xi \,d\xi d\eta\over\left[(x-\xi)^2+(\eta-y)^2\right]^{1-{{\rm Re\, m}\over 2}}}
$$
$$
=\O(N^2)\int_{\eta={N\over2}}^N { d\eta\over (\eta-y)^{2-{\rm Re\, m}}}
$$
$$
=\O(N^2)\left[{1\over(N-y)^{1-{\rm Re\,}m}}-{1\over({N\over2}-y)^{1-{\rm Re\,}m}}\right]=\O(N^{1+{\rm Re\,}m}).
$$
Then, thanks to (\ref{Fmxi}), we have
$$
\int_{D_6}\left|{\d F_m\over \d \xi}\right|d\xi d\eta=\O(1)\int_{\xi={N\over 2}}^{\xi={N}}\int_{\eta={N\over2}}^{\eta={N}}\left[{1\over\left[\left(x-\xi\right)^2+(\eta-y)^2\right]^{1-{\Re m\over2}}} \right.\qquad\qquad\qquad\qquad
$$
$$
\qquad\qquad\qquad\qquad\qquad\qquad\qquad\qquad \left. +{\xi(x+\xi)\over\left(\left(x-\xi\right)^2+(\eta-y)^2\right)^{2-{\Re m\over2}}}\right]d\xi d\eta
$$
$$
=\O(1)\int_{\xi={N\over 2}}^{\xi={N}}\int_{\eta={N\over2}}^{\eta={N}}\left[{1\over(\eta-y)^{2-{\Re m}}} +{\xi(x+\xi)\over(\eta-y)^{4-{\Re m}}}\right]d\xi d\eta
$$
$$
=\O(N)+\O(N^3)\int_{\eta={N\over2}}^N{d\eta\over(\eta-y)^{4-{\rm Re}\,m}}=\O(N)+\O(N^{\Re m})=\O(N).
$$
The estimate (\ref{Fmeta}) gives
$$
\int_{D_6}\left|{\d F_m\over \d \eta}\right|d\xi d\eta=\O(1)\int_{\xi={N\over 2}}^{\xi={N}}\int_{\eta={N\over2}}^{\eta={N}}
{\xi\,d\xi\,d\eta \over|\eta-y|^{3-{\rm Re\,}m}}=O(N^2)\int_{\eta={N\over2}}^{\eta={N}}
{d\eta \over|\eta-y|^{3-{\rm Re\,}m}}
$$
$$
=\O(N^{{\rm Re\,}m}).
$$

\bigskip$\underline{\hbox{Sur }D_7}$ : this case is analogous to the case $D_6$.

\bigskip$\underline{\hbox{Sur }D_8}$ : this case is analogous to the case $D_5$.
\end{proof}

\begin{lem}
For Re $m\geq1$, all estimations obtained on the Lemma \ref{estimationMleq1} are true.
\end{lem}
\begin{proof}
For Re $m\geq 1$, we have 
$$
F_m(x,y,\xi,\eta)=-{\xi^m\over 2\pi}\int_{\theta=0}^\pi \sin^{m-1}\theta \left({1\over[(x-\xi)^2+4x\xi\sin^2{\theta\over2}+(y-\eta)^2]^{m/2}}\right.\qquad\qquad\qquad\qquad
$$
$$
\qquad\qquad\qquad\qquad\qquad\qquad\qquad\qquad\qquad \left. -{1\over[(x+\xi)^2-4x\xi\sin^2{\theta\over2}+(y-\eta)^2]^{m/2}}\right)d\theta.
$$
Since for all $(\xi,\eta)\in\H^+$, we have
$$
\left|\left[(x+\xi)^2-4x\xi\sin^2{\theta\over2}+(y-\eta)^2\right]^{m/2}\right|=\left|\left[x^2+\xi^2+2x\xi\cos\theta+(y-\eta)^2\right]^{m/2}\right|,
$$
then for all $(\xi,\eta)\in\H^+$,
\begin{equation}\label{mino}
\left|\left[(x+\xi)^2-4x\xi\sin^2{\theta\over2}+(y-\eta)^2\right]^{m/2}\right|\geq\left((x-\xi)^2+(y-\eta)^2\right)^{{{\rm Re\,}m\over 2}}
\end{equation}
and there is a constant $C'_1$ such that for all $(\xi,\eta)\in\H^+$, we have
\begin{equation}\label{fm}
|F_m|\leq  {C'_1\xi^{{\rm Re\,}m}\over \left((x-\xi)^2+(y-\eta)^2\right)^{{{\rm Re\,}m\over 2}}}.
\end{equation}
%
%
This inequality does not suffice to estimate integrals supported on  $D_1$.  We can improve inequality (\ref{fm}) as follows :

We rewrite $F_m$ as
$$
F_m(x,y,\xi,\eta)=-{\xi^m\over 2\pi}\int_{\theta=0}^\pi \sin^{m-1}\theta K_m(x,y,\xi,\eta,\theta)d\theta
$$
where
$$
K_m(x,y,\xi,\eta,\theta)={1\over[(x-\xi)^2+4x\xi\sin^2{\theta\over2}+(y-\eta)^2]^{m/2}}\qquad\qquad\qquad\qquad\qquad\qquad
$$
$$
\qquad\qquad\qquad\qquad\qquad\qquad\qquad\qquad\qquad\quad-{1\over[(x+\xi)^2-4x\xi\sin^2{\theta\over2}+(y-\eta)^2]^{m/2}}.
$$

For $(x,y)\in\H^+$ fixed, $\theta\in[0,\pi]$ fixed and $\eta\in\R$ fixed, we define the function $g_m$ on $[-1/N,1/N]$ with $1/N<x$ by
$$
g_m(\xi)={1\over[(x-\xi)^2+4x\xi\sin^2{\theta\over2}+(y-\eta)^2]^{m/2}}.
$$
This function is well defined because 
$$
(x-\xi)^2+4x\xi\sin^2{\theta\over2}+(y-\eta)^2=x^2+\xi^2-2x\xi\cos\theta+(y-\eta)^2\geq(x-|\xi|)^2+(y-\eta)^2
$$
and this last term is larger than $(x-1/N)^2>0$.

We have
$$
K_m(x,y,\xi,\eta,\theta)=g_m(\xi)-g_m(-\xi)
$$
thus
$$
|K_m(x,y,\xi,\eta,\theta)|\leq 2\xi\sup_{[-\xi,\xi]}|g_m'|\leq2|m|\xi{|\xi-x|+2x\over[(x-\xi)^2+(y-\eta)^2]^{1+{1\over 2}\Re m}},
$$
which implies that there exists a constant $c'_1$ such that
\begin{equation}\label{fmbis}
\forall(\xi,\eta)\in D_1,\qquad |F_m|\leq c'_1{\xi^{\Re m+1}\over\left[\left(x-\xi \right)^2+(\eta-y)^2\right]^{1+{1\over2}\Re m}}.
\end{equation}

Similarly, we have
$$
{\d F_m\over \d\xi}={m\, F_m\over\xi}+{m\xi^m\over 2\pi}\int_{\theta=0}^\pi \sin^{m-1}\theta \left({(\xi-x)+2x\sin^2{\theta\over2}\over[(x-\xi)^2+4x\xi\sin^2{\theta\over2}+(y-\eta)^2]^{{m\over2}+1}}\right.\qquad\qquad\qquad\qquad
$$
\begin{equation}\label{dFmxi}
\qquad\qquad\qquad\qquad\qquad\qquad\qquad\qquad\qquad \left. -{(\xi+x)-2x\sin^2{\theta\over2}\over[(x+\xi)^2-4x\xi\sin^2{\theta\over2}+(y-\eta)^2]^{{m\over2}+1}}\right)d\theta.
\end{equation}
and as before,
$$
\forall \theta\in[0,\pi],\quad \left|{[(\xi-x)+2x\sin^2{\theta\over2}]\sin^{m-1}\theta\over[(x-\xi)^2+4x\xi\sin^2{\theta\over2}+(y-\eta)^2]^{{m\over2}+1}}\right|\leq { \left|(\xi-x)+2x\sin^2{\theta\over2}\right|\over[(x-\xi)^2+(y-\eta)^2]^{{{\rm Re\,} m\over2}+1}}
$$
$$
={|\xi-x\cos\theta|\over [(x-\xi)^2+(y-\eta)^2]^{{{\rm Re\,} m\over2}+1} }\leq{\xi+x\over [(x-\xi)^2+(y-\eta)^2]^{{{\rm Re\,} m\over2}+1} }
$$
and thanks to (\ref{mino}) :
 $$
\forall \theta\in[0,\pi],\quad \left|{[(\xi+x)-2x\sin^2{\theta\over2}]\sin^{m-1}\theta\over[(x+\xi)^2-4x\xi\sin^2{\theta\over2}+(y-\eta)^2]^{{m\over2}+1}}\right|\leq { \left|(\xi+x)-2x\sin^2{\theta\over2}\right|\over[(x-\xi)^2+(y-\eta)^2]^{{{\rm Re\,} m\over2}+1}}
$$
$$
={|\xi+x\cos\theta|\over [(x-\xi)^2+(y-\eta)^2]^{{{\rm Re\,} m\over2}+1} }\leq{\xi+x\over [(x-\xi)^2+(y-\eta)^2]^{{{\rm Re\,} m\over2}+1} }.
$$
Those estimations, the formula (\ref{dFmxi}) and the inequality (\ref{fm}) show that there is a constant $C'_2$ such that large enough $N$ and for all $(\xi,\eta)\in\H^+$, we have
\begin{equation}\label{fmxi}
\left|{\d F_m\over \d\xi}\right|\leq C'_2\left({\xi^{{\rm Re\,}m-1}\over [(x-\xi)^2+(y-\eta)^2]^{{{\rm Re\,} m\over2}} }+{\xi^{{\rm Re\,}m}(\xi+x)\over [(x-\xi)^2+(y-\eta)^2]^{{{\rm Re\,} m\over2}+1} }\right).
\end{equation}

We can improve this inequality on $D_1$, for this, we need to use the inequality (\ref{fmbis}) instead of (\ref{fm}) and we obtain that there is two constants $C''_2$ and $C'''_2$ (which do not depend of $N$)
such that for all $(\xi,\eta)\in D_1$
$$
\left|{\d F_m\over \d\xi}\right|\leq C''_2\left({\xi^{{\rm Re\,}m}\over [(x-\xi)^2+(y-\eta)^2]^{1+{{\rm Re\,} m\over2}} }+{\xi^{{\rm Re\,}m}(\xi+x)\over [(x-\xi)^2+(y-\eta)^2]^{{{\rm Re\,} m\over2}+1} }\right)
$$
\begin{equation}\label{fmxibis}
\leq C_2'''{\xi^{{\rm Re\,}m}\over [(x-\xi)^2+(y-\eta)^2]^{1+{{\rm Re\,} m\over2}} }
\end{equation}

%
%
%
%

Finally,
$$
{\d F_m\over \d\eta}={m(\eta-y)\xi^m\over 2\pi}\int_{\theta=0}^\pi \sin^{m-1}\theta \left({1\over[(x-\xi)^2+4x\xi\sin^2{\theta\over2}+(y-\eta)^2]^{{m\over 2}+1}}\right.\qquad\qquad\qquad\qquad
$$
$$
\qquad\qquad\qquad\qquad\qquad\qquad\qquad\qquad\qquad \left. -{1\over[(x+\xi)^2-4x\xi\sin^2{\theta\over2}+(y-\eta)^2]^{{m\over 2}+1}}\right)d\theta.
$$
Similarly, there is a constant $C'_3$ such that for all large enough $N$ and for all $(\xi,\eta)\in\H^+$, we have
\begin{equation}\label{fmeta}
\left|{\d F_m\over \d\eta}\right|\leq C'_3 {|\eta-y|\xi^{{\rm Re\,}m}\over \left((x-\xi)^2+(y-\eta)^2\right)^{{{\rm Re\,}m\over 2}+1}}.
\end{equation}
Thanks to those inequalities, we will estimate the integrals of those functions on each domain $D_i$.

\bigskip$\underline{\hbox{On }D_1}$ :
due to (\ref{fmbis}), we have
$$
\int_{D_1}|F_m|d\xi d\eta= \O(1)\int_{\xi={1\over 2N}}^{\xi={1\over N}}\int_{\eta=-{N\over2}}^{\eta={N\over2}} {\xi^{\Re m+1}d\xi d\eta\over\left[\left(x-\xi \right)^2+(\eta-y)^2\right]^{1+{1\over2}\Re m}}
$$
$$
=\O(N)\int_{\xi={1\over 2N}}^{\xi={1\over N}}\xi^{{\rm Re\,}m+1}d\xi=\O(N)\left[\left({1\over N}\right)^{{\rm Re\,}m+2}-\left({1\over 2N}\right)^{{\rm Re\,}m+2}\right]
$$
$$
=\O(1/N^{{\rm Re\,}m+1}).
$$
Then thanks to (\ref{fmxibis}), we have
$$
\int_{D_1}\left|{\d F_m\over \d\xi}\right|d\xi d\eta= \O(1)\int_{\xi={1\over 2N}}^{\xi={1\over N}}\int_{\eta=-{N\over2}}^{\eta={N\over2}}{\xi^{\Re m} d\xi d\eta\over\left[\left(x-\xi\right)^2+(\eta-y)^2\right]^{1+{1\over2}\Re m}}.
$$
$$
=\O(N)\int_{\xi={1\over 2N}}^{\xi={1\over N}}\xi^{{\rm Re\,}m}d\xi=\O(1/N^{{\rm Re\,}m}).
$$

\bigskip$\underline{\hbox{On }D_2}$ : due to (\ref{fm}), we have
$$
\int_{D_2}|F_m|d\xi d\eta=\O(1)\int_{\xi={1\over N}}^{{N\over 2}}\int_{\eta={N\over2}}^N  {\xi^{{\rm Re\,}m}d\xi d\eta\over \left((x-\xi)^2+(y-\eta)^2\right)^{{{\rm Re\,}m\over 2}}}
$$
$$
=\O(1)\int_{\xi={1\over N}}^{{N\over 2}}\int_{\eta={N\over2}}^N  {\xi^{{\rm Re\,}m}d\xi d\eta\over |y-{N\over2}|^{{\rm Re\,}m}}=\O(N^2),
$$
because we integrate a bounded function (independently of $N$) on a domain with measure controlled by $\O(N^2)$.

Then, the inequality (\ref{fmeta}) implies
$$
\int_{D_2}\left|{\d F_m\over \d\eta}\right|d\xi d\eta= \O(1)\int_{\xi={1\over N}}^{{N\over 2}}\int_{\eta={N\over2}}^N  {|\eta-y|\xi^{{\rm Re\,}m} d\xi d\eta\over \left((x-\xi)^2+(y-\eta)^2\right)^{{{\rm Re\,}m\over 2}+1}}
$$
$$
=\O(1)\int_{\xi={1\over N}}^{{N\over 2}}\int_{\eta={N\over2}}^N  {\xi^{{\rm Re\,}m} d\xi d\eta\over |y-\eta|^{{\rm Re\,}m+1}}=
\O(1)\int_{\xi={1\over N}}^{{N\over 2}}\int_{\eta={N\over2}}^N  {N^{{\rm Re\,}m} d\xi d\eta\over |{N\over 2}-y|^{{\rm Re\,}m+1}}=\O(N).
$$

\bigskip$\underline{\hbox{On }D_3}$ : due to (\ref{fm}), we have
$$
\int_{D_3}|F_m|d\xi d\eta=\O(1)\int_{\xi={N\over2}}^{{N}}\int_{\eta=-{N\over2}}^{N\over2}  {\xi^{{\rm Re\,}m} d\xi d\eta\over \left((x-\xi)^2+(y-\eta)^2\right)^{{{\rm Re\,}m\over 2}}}
$$
$$
=\O(1)\int_{\xi={N\over2}}^{{N}}\int_{\eta=-{N\over2}}^{N\over2}  {\xi^{{\rm Re\,}m} d\xi d\eta\over \left((x-{N\over 2})^2+(y-\eta)^2\right)^{{{\rm Re\,}m\over 2}}}
$$
$$
=\O(N^{{\rm Re\,}m+1})\int_{\eta=-{N\over2}}^{N\over2} { d\eta\over \left((x-{N\over 2})^2+(y-\eta)^2\right)^{{{\rm Re\,}m\over 2}}}=\O(N^2).
$$
Then, thanks to (\ref{fmxi}), we have
$$
\int_{D_3}\left|{\d F_m\over\d\xi}\right|d\xi d\eta=\O(1)\int_{\xi={N\over2}}^{{N}}\int_{\eta=-{N\over2}}^{N\over2}  \left({\xi^{{\rm Re\,}m-1}\over [(x-\xi)^2+(y-\eta)^2]^{{{\rm Re\,} m\over2}} }\right. \qquad\qquad\qquad\qquad\qquad
$$
$$
\qquad\qquad\qquad\qquad\qquad\qquad\qquad\qquad \left.+{\xi^{{\rm Re\,}m}(\xi+x)\over [(x-\xi)^2+(y-\eta)^2]^{{{\rm Re\,} m\over2}+1} }\right)d\xi d\eta
$$
$$
=\O(N^{{\rm Re\,}m})\int_{\eta=-{N\over2}}^{N\over2}{d\eta\over [(x-{N\over2})^2+(y-\eta)^2]^{{{\rm Re\,} m\over2}} } \qquad\quad
$$
$$
\qquad\qquad\qquad\qquad\qquad\qquad\qquad+\O(N^{{\rm Re\,}m+2})\int_{\eta=-{N\over2}}^{N\over2}{d\eta\over [(x-{N\over2})^2+(y-\eta)^2]^{{{\rm Re\,} m\over2}+1} }
$$
$$
=\O(N)+\O(N)=\O(N).
$$

\bigskip$\underline{\hbox{On }D_4}$ : this case is analogous to the case $D_2$.

\bigskip$\underline{\hbox{On }D_5}$ : due to (\ref{fm}), we have
$$
\int_{D_5}|F_m|d\xi d\eta=\O(1)\int_{\xi={1\over 2N}}^{{1\over N}}\int_{\eta={N\over2}}^N  {\xi^{{\rm Re\,}m}d\xi d\eta\over \left((x-\xi)^2+(y-\eta)^2\right)^{{{\rm Re\,}m\over 2}}}
$$
$$
=\O(1/N^{{\rm Re\,}m+1})\int_{\eta={N\over2}}^N  { d\eta\over |y-{N\over2}|^{{\rm Re\,}m}}=\O(1/N^{2{\rm Re\,}m}).
$$
Then, thanks to (\ref{fmxi}), we have
$$
\int_{D_5}\left|{\d F_m\over \d \xi}\right|d\xi d\eta=\O(1)\int_{\xi={1\over 2N}}^{\xi={1\over N}}\int_{\eta={N\over2}}^{\eta={N}} \left({\xi^{{\rm Re\,}m-1}\over [(x-\xi)^2+(y-\eta)^2]^{{{\rm Re\,} m\over2}} }\right. \qquad \qquad \qquad  \qquad
$$
$$
 \qquad \qquad \qquad \qquad \qquad \qquad \qquad \qquad \qquad \left.+{\xi^{{\rm Re\,}m}(\xi+x)\over [(x-\xi)^2+(y-\eta)^2]^{{{\rm Re\,} m\over2}+1} }\right)d\xi d\eta
$$
$$
=\O(1)\int_{\xi={1\over 2N}}^{\xi={1\over N}}\int_{\eta={N\over2}}^{\eta={N}}\left({\xi^{{\rm Re\,}m-1}\over|y-\eta|^{{\rm Re\,} m} }+{\xi^{{\rm Re\,}m}(\xi+x)\over |y-\eta|^{{\rm Re\,} m+2} }\right)d\xi d\eta
$$
$$
=\O(1/N^{2{\rm Re\,}m-1}).
$$
With the inequality (\ref{fmeta}), we have
$$
\int_{D_5}\left|{\d F_m\over \d \eta}\right|d\xi d\eta=\O(1)\int_{\xi={1\over 2N}}^{\xi={1\over N}}\int_{\eta={N\over2}}^{\eta={N}} {|\eta-y|\xi^{{\rm Re\,}m} d\xi d\eta\over \left((x-\xi)^2+(y-\eta)^2\right)^{{{\rm Re\,}m\over 2}+1}}
$$
$$
=\O(1)\int_{\xi={1\over 2N}}^{\xi={1\over N}}\int_{\eta={N\over2}}^{\eta={N}} {\xi^{{\rm Re\,}m} d\xi d\eta\over |y-\eta|^{{\rm Re\,}m+1}}=\O(1/N^{2{\rm Re\,}m+1})
$$

\bigskip$\underline{\hbox{On }D_6}$ : due to (\ref{fm}), we have
$$
\int_{D_6}|F_m|d\xi d\eta=\O(1)\int_{\xi={N\over2}}^{{N}}\int_{\eta={N\over2}}^N {\xi^{{\rm Re\,}m}d\xi d\eta\over \left((x-\xi)^2+(y-\eta)^2\right)^{{{\rm Re\,}m\over 2}}}
$$
$$
=\O(N^{{\rm Re\,}m+1})\int_{\eta={N\over2}}^N {d\eta\over ({N\over2}-y)^{{\rm Re\,}m}}=\O(N^2).
$$
Then, thanks to (\ref{fmxi}), we obtain
$$
\int_{D_6}\left|{\d F_m\over \d \xi}\right|d\xi d\eta=\O(1)\int_{\xi={N\over 2}}^{\xi={N}}\int_{\eta={N\over2}}^{\eta={N}} \left({\xi^{{\rm Re\,}m-1}\over [(x-\xi)^2+(y-\eta)^2]^{{{\rm Re\,} m\over2}} }\right. \qquad\qquad \qquad
$$
$$
\qquad\qquad\qquad\qquad\qquad\qquad   \left.+{\xi^{{\rm Re\,}m}(\xi+x)\over [(x-\xi)^2+(y-\eta)^2]^{{{\rm Re\,} m\over2}+1} }\right)d\xi d\eta
$$
$$
=\O(1)\int_{\xi={N\over 2}}^{\xi={N}}\int_{\eta={N\over2}}^{\eta={N}} \left({\xi^{{\rm Re\,}m-1}\over |y-\eta|^{{\rm Re\,} m} }+{\xi^{{\rm Re\,}m}(\xi+x)\over|y-\eta|^{{\rm Re\,} m+2} }\right)d\xi d\eta
$$
$$
=\O(N)+\O(N^{{\rm Re\,}m+2})\int_{\eta={N\over2}}^{\eta={N}} {d\eta\over |y-\eta|^{{\rm Re\,} m+2} }=\O(N).
$$
Finally, the inequality (\ref{fmeta}) implies
$$
\int_{D_6}\left|{\d F_m\over \d \eta}\right|d\xi d\eta=\O(1)\int_{\xi={N\over 2}}^{\xi={N}}\int_{\eta={N\over2}}^{\eta={N}} {|\eta-y|\xi^{{\rm Re\,}m}d\xi d\eta\over \left((x-\xi)^2+(y-\eta)^2\right)^{{{\rm Re\,}m\over 2}+1}}
$$
$$
=\O(N^{{\rm Re\,}m+1})\int_{\eta={N\over2}}^{\eta={N}} {d\eta\over |y-\eta|^{{\rm Re\,}m+1}}=\O(N).
$$

\bigskip$\underline{\hbox{On }D_7}$ : this case is analogous to the case $D_6$.

\bigskip$\underline{\hbox{On }D_8}$ : this case is analogous to the case $D_5$.

\end{proof}

In the following table, we summarize the results obtained on the previous lemmas :

\bigskip
\resizebox{12.5cm}{!}{
 \setlength{\extrarowheight}{5 pt}
\begin{tabular}{|l||*{5}{c|}}\hline
\makebox[1em]{$i$}
&\makebox[6em]{$\mathop{\sup}_{D_i}|L_{-m}\phi_N|$}&\makebox[6em]{$\int_{D_i}|F_m|d\xi d\eta$}&\makebox[5em]{$(|\d_{\xi}\phi_N|,|\d_{\eta}\phi_N|)$}
&\makebox[5em]{$\int_{D_i}|\d_{\xi}F_m|$}&\makebox[5em]{$\int_{D_i}|\d_{\eta}F_m|$}\\\hline\hline
1 & $\O(N^2)$     & $\O(1/N^2)$  &  $(\O(N),0)     $                                    & $\O({1\over N})$   & $\times$\\ \hline
2 & $\O(1/N^2)$  & $\O(N^2)$     &  $(0,\O({1\over N})) $                         &             $\times$        & $\O(N)$ \\ \hline
3 & $\O(1/N^2)$  & $\O(N^2)$     &  $(\O({1\over N}),0) $                         &            $\O(N)$          & $\times$\\ \hline
4 & $\O(1/N^2)$  & $\O(N^2)$     &  $(0,\O({1\over N})) $                         &            $\times$         & $\O(N)$ \\ \hline
5 & $\O(N^2)$     & $\O(1/N^2)$  &  $(\O(N),\O({1\over N}))$                   & $\O({1\over N})$   & $\O({1\over N^2})$\\ \hline
6 & $\O(1/N^2)$  & $\O(N^2)$     &  $(\O({1\over N}),\O({1\over N}))$    &         $\O(N)$             &  $\O(N)$ \\ \hline
7 & $\O(1/N^2)$  & $\O(N^2)$     &  $(\O({1\over N}),\O({1\over N}))$    &         $\O(N)$             &  $\O(N)$ \\ \hline
8 & $\O(N^2)$     & $\O(1/N^2)$  &  $(\O(N),\O({1\over N}))$                   & $\O({1\over N})$   & $\O({1\over N^2})$\\ \hline
\end{tabular}

 \setlength{\extrarowheight}{0 pt}}

\bigskip\noindent
We can easily check that for each $i\in\{1,\, \ldots,\, 8\}$, the quantities 
$$
\mathop{\sup}_{D_i}|L_{-m}\phi_N|\int_{D_i}|F_m|,\quad \mathop{\sup}_{D_i}|\d_{\xi}\phi_N|\int_{D_i}|\d_{\xi}F_m|\quad{\rm and}\,\, \mathop{\sup}_{D_i}|\d_{\eta}\phi_N|\int_{D_i}|\d_{\eta}F_m|
$$
stay bounded. Therefore, 
$$
u(x,y)=o(1)
$$
when $N\to+\infty$. Thus 
$$
u\equiv 0
$$
and this completes the proof of the Proposition \ref{Unicite}.
\end{proof}

\begin{lem}\label{Lemme1} 
Let $u\in\mathcal{D}(\H^+)$ and let $(x,y)\in\H^+$, we define
$$
U(x,y)=\int_{\H^+}u(\xi,\eta)F_m(x,y,\xi,\eta)d\xi\,d\eta,
$$
then  $\displaystyle\lim_{\|(x,y)\|\to+\infty} U=0$, and for all $y\in\R$, $\displaystyle\lim_{(0,y)} U=0$. 

Moreover, $U\in C^\infty(\H^+\setminus{ \rm supp}\,u)$ and for all $(x,y)\not\in${ \rm supp }$u$, we have $L_{m,x,y}U(x,y)=0$.
\end{lem}

\begin{proof}
When $(\xi,\eta)$ is fixed, and since
$$
F_m(x,y,\xi,\eta)=-{\xi x^{1-m}\over 2\pi}\int_{\theta=0}^{\pi}{\sin^{1-m}\theta \, d\theta \over \left[(x-\xi)^2+4x\xi\sin^2\left({\theta\over 2}\right)+(y-\eta)^2\right]^{1-{m\over2}}}
$$
for Re $m<1$, then $\displaystyle F_m(x,y,\xi,\eta)\mathop{\longrightarrow}_{\|(x,y)\|\to+\infty}0$
and the first result of the lemma is shown.

Similarly, if Re $m\geq 1$, 
$$
F_m(x,y,\xi,\eta)=-{\xi^m\over 2\pi}\int_{\theta=0}^{\pi}\sin^{m-1}\theta \left[{1\over\left[(x-\xi)^2+4x\xi\sin^2\left({\theta\over 2}\right)+(y-\eta)^2\right]^{m\over2}}-\right.\qquad\qquad\qquad\qquad$$
$$
\quad\qquad\qquad\qquad\qquad\qquad\qquad\qquad\qquad\left.-{1\over\left[(x+\xi)^2-4x\xi\sin^2\left({\theta\over 2}\right)+(y-\eta)^2\right]^{m\over2}}\right]d\theta
$$
then $\displaystyle F_m(x,y,\xi,\eta)\mathop{\longrightarrow}_{\|(x,y)\|\to+\infty}0$
and the first result of the lemma is shown.

For the second result, it suffices to see, for Re $m<1$, that
$$
F_m(x,y,\xi,\eta)\mathop\sim_{(x,y)\to(0,y')}-\frac{\xi x^{1-m}}{2\pi[\xi^2+(y'-\eta)^2]^{1-m/2}}\int_0^\pi\sin^{1-m}\theta\,d\theta
$$
which implies the desired result.

Now, we assume that Re $m\geq 1$. Let $(\xi,\eta)$ be fixed in the support of $u$, which is a compact set of $\H^+$. In particular, there exist $M>0$ and $\alpha>0$ which do not depend of $u$ such that  $\|(\xi,\eta)\|\leq M$ et $\xi\geq2\alpha$. Let $y$ be in $\R$.

By denotting for $x\in[-\alpha,\alpha]$,
$$
f_m(x)={1\over\left[(x-\xi)^2+4x\xi\sin^2\left({\theta\over 2}\right)+(y-\eta)^2\right]^{m\over2}},
$$
By the mean value inequality, for $x>0$ near 0, we have
$$
|f_m(x)-f_m(0)|\leq x \sup_{[0,\alpha]}|f_m'|
$$
and
$$
|f_m(-x)-f_m(0)|\leq x\sup_{[-\alpha,0]}|f_m'|.
$$
then
$$
|f_m(x)-f_m(-x)|\leq 2x\sup_{[-\alpha,\alpha]}|f_m'|\leq2x|m|{3M+\alpha\over \alpha^{\Re m+2}}.
$$
In particular,
$$
\sup_{(\xi,\eta)\in{\rm supp\,} u\atop y\in\R} |F_m(x,y)|=\O(x)
$$
when $x\to0+$. The second result is proved.

The last result can be deduced of the fact that if $(x,y)\not=(\xi,\eta)$ are both in $\H^+$, then
$$
L_{m,x,y}F_m(x,y,\xi,\eta)=0.
$$
\end{proof}

\begin{rem}
If $U\in \mathcal{D}(\H^+)$, then $L_{m,x,y}U=u$, but this identity is not necessary true if $U\not\in\mathcal{D}(\H^+)$. In particular, we can not say that in the Lemma \ref{Lemme1}, we have $L_mU=u$.
\end{rem}

Now, we will prove a decomposition theorem for axisymmetric potentials, it is interesting to compare the following theorem to known result in \cite[Theorem 2 section 4]{BFL}  (the fundamental difference is that in this work, the conductivity is not extended in all domain by reflection through the boundary $\d\Omega$).

\begin{thm}
Let $m\in\C$. Let $\Omega$ be an open set of $\H^+$ and let $K$ be a compact set of $\Omega$. If $u\in{C}^2(\Omega\setminus K)$ satisfies $L_mu=0$ in $\Omega\setminus K$, then $u$ has a unique decomposition as follows :
$$
u=v+w,
$$
where $v\in{C}^2(\Omega)$ satisfies $L_mv=0$ in $\Omega$ and $w\in {C}^2(\H^+\setminus K)$ satisfies $L_mw=0$ in $\H^+\setminus K$ with $\displaystyle \lim_{\d\H^+} w=0$.
\end{thm}

\begin{proof}
For $E\subset\C$ and $\rho>0$, we define $E_\rho=\{x\in\C,\ d(x,E)<\rho\}$ ($E_\rho$ is a neighborhood of $E$).

At first, we assume that $\Omega$ is a relatively compact open set of $\H^+$. We choose $\rho$ as small as $K_\rho$ and $(\d\Omega)_\rho$ are disjoint. There is a function $\varphi_\rho\in\mathcal{D}(\H^+)$ compactly supported on $\Omega\setminus K$ such that $\varphi_\rho\equiv 1$ in a neighborhood of $\Omega\setminus(K_\rho\cup(\d\Omega)_\rho)$.

\hfil
\resizebox{10cm}{!}{
\begin{tikzpicture}

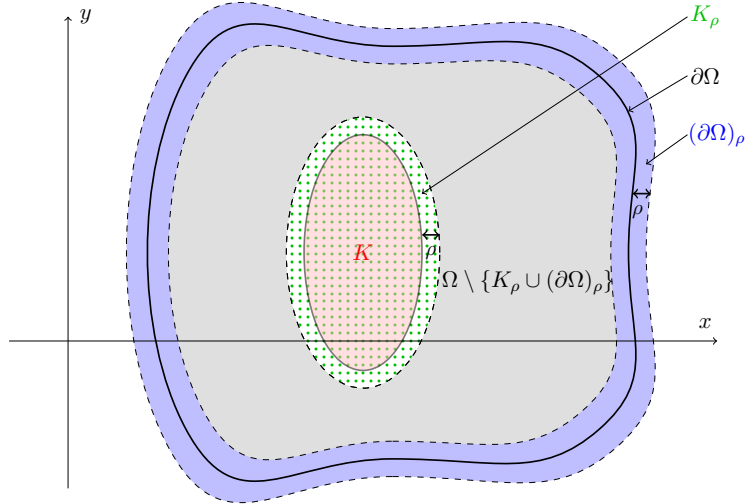
\draw (0.5,-4.8) node{Figure : $\varphi_\rho\equiv1$ on the gray domain};

\draw[fill=white!75!blue,dashed] plot[smooth cycle,tension=0.9]coordinates
{(-3.3,-3.3)(0.5,-3.8)(4.3,-3.3)(4.8,0)(4.3,3.3)(0.5,3.8)(-3.3,3.3)};
\draw[fill=white!75!gray,dashed] plot[smooth cycle,tension=0.9]coordinates
{(-2.7,-2.7)(0.5,-3.2)(3.7,-2.7)(4.2,0)(3.7,2.7)(0.5,3.2)(-2.7,2.7)};

\draw[fill=white,dashed] (0,0) ellipse (1.3 and 2.3);
\draw[dashed, pattern=dots, pattern color=green!75!black] (0,0) ellipse (1.3 and 2.3);

\draw[thick,fill=white!75!red,opacity=0.5] (0,0) ellipse (1 and 2);

\draw[thick] plot[smooth cycle,tension=0.9]coordinates
{(-3,-3)(0.5,-3.5)(4,-3)(4.5,0)(4,3)(0.5,3.5)(-3,3)};

\draw[->] (-6,-1.5) -- (6,-1.5);
\draw[->] (-5,-4) -- (-5,4);
\draw (5.8,-1.2) node {$x$};
\draw (-4.7,4) node {$y$};

\draw[red] (0,0) node {$K$};
\draw[green!75!black] (5.8,4) node {$K_\rho$};
\draw (2.8,-0.5) node {$\Omega\setminus\{K_\rho\cup (\d\Omega)_\rho\}$};
\draw (5.8,3) node {$\d\Omega$};
\draw[blue] (6,2) node {$(\d\Omega)_\rho$};
\draw[->] (5.5,3)--(4.5,2.40);
\draw[->] (5.5,2)--(4.8,1.5);
\draw[->] (5.5,4)--(1,1);
\draw[<->,thick] (4.57,1)--(4.88,1);
\draw[<->,thick] (1,0.3)--(1.3,0.3);
\draw[thick] (4.65,0.7) node {{\small $\rho$}};
\draw[thick] (1.15,0) node {{\small $\rho$}};
\end{tikzpicture}}
\hfill

For $z=x+iy\in\Omega\setminus(K_\rho\cup(\d\Omega)_\rho)$, we denote
$$
F_z(\zeta):=F_m(x,y,\xi,\eta)\quad\hbox{and}\quad L_\zeta:=L_{m,\xi,\eta}\qquad {\rm for }~ \zeta=\xi+i\eta,
$$
Thanks to Proposition \ref{RepresentationIntegrale}, we have 
$$
u(z)=u\varphi_\rho(z)=\int_{\Omega_\rho}F_z(\zeta)L_\zeta(u\varphi_\rho)(\zeta)d\xi d\eta
$$
$$
=\int_{(\d\Omega)_\rho}F_z(\zeta)L_\zeta(u\varphi_\rho)(\zeta)d\xi d\eta +
\int_{K_\rho}F_z(\zeta)L_\zeta(u\varphi_\rho)(\zeta)d\xi d\eta 
$$
$$
=v_\rho(z)+w_\rho(z).
$$
Then, the last result of Lemma \ref{Lemme1} shows us that $v_\rho$ satisfies $L_mv_\rho=0$ on $\Omega\setminus(\d\Omega)_\rho$ and $w_\rho$ satisfies  $L_mw_\rho=0$ on $\H^+\setminus K_\rho$. We also have $\lim_{\d\H^+}w_\rho=0$.

Now, we assume that $\sigma<\rho$. As previously, we obtain the decomposition $u=v_\sigma+
w_\sigma$ on $\Omega\setminus(K_\sigma\cup(\d\Omega)_\sigma)$. We claim that $v_\rho=v_\sigma$ on $\Omega\setminus(\d\Omega)_\rho $ and $w_\rho=w_\sigma$ on $\H^+\setminus K_\rho$. To see this, note that if $z\in\Omega\setminus(K_\rho\cup(\d\Omega)_\rho)$, then $v_\rho(z)+w_\rho(z)=v_\sigma(z)+w_\sigma(z)$.   

We will designate by $(1)$ the Weinstein equation $L_mu=0$. Thus $w_\rho-w_\sigma$ satisfies $(1)$ on $\H^+\setminus K_\rho$, which is equal to $v_\sigma-v_\rho$ on $\Omega\setminus(K_\rho\cup(\d\Omega)_\rho)$, therefore $v_\sigma-v_\rho$ extends to a solution of $(1)$ on $\Omega\setminus(\d\Omega)_\rho$. 

Finally, $w_\rho-w_\sigma$ extends to a solution of $(1)$ on $\H^+$, and $\displaystyle\lim_{\d\H^+}{w_\rho-w_\sigma}=0$. Due to Proposition \ref{Unicite}, we have
$$
w_\rho=w_\sigma,
$$
and then $v_\rho=v_\sigma$.

For $z\in\Omega$, we can define $v(z)=v_\rho(z)$ for $\rho$ as small as $z\in\Omega\setminus(\d\Omega)_\rho$. Similarly, for $z\in\H^+\setminus K$, we put $w(z)=w_\rho(z)$ for small $\rho$. We have proved the desired decomposition $u=v+w$.

Now, assume that $\Omega$ is an arbitrary domain of $\H^+$ and let $u$ be a solution of $L_mu=0$ on $\Omega\setminus K$. We choose $a\in\H^+$ and $R$ large enough so that $K\subset D(a,R)$ and $D(a,R)$ be a relatively compact set of $\H^+$. Let $\omega=\Omega\cap D(a,R)$. Note that $K$ is a compact set of $\omega$ which is a relatively compact open set of $\H^+$ and  $u$ satisfies $(1)$ on $\omega\setminus K$. Applying the results demonstrated for relatively compact open sets, we obtain
$$
u(z)=\tilde v(z)+\tilde w(z)
$$
for $z\in\omega\setminus K$, where $\tilde v$ satisfies $(1)$  on $\omega$ and $\tilde w$ satisfies $(1)$ on $\H^+\setminus K$ with $\lim_{\d\H^+}\tilde w=0$. Note that $V=u-\tilde w$ satisfies $(1)$ on $\Omega\setminus K$ and $V$ can be extended into a solution of $(1)$ in a neighborhood of $K$ because $V=\tilde v$ on $\omega$. The sum $u=V+\tilde w$ provides us a desired decomposition of $u$.

As before, if we have another decomposition $u=v+w$ with $v\in{C}^2(\Omega)$, $L_mv=0$ and with $w\in{C}^2(\H^+\setminus K)$, $L_mw=0$ and $\lim_{\d\H^+}w=0$, then we have $V-v=w-\tilde w$ on $\Omega\setminus K$. The function $w-\tilde w$ can be extended on $\H^+$ into a solution of $L_m(w-\tilde w)=0$ on $\H^+$ with $\lim_{\d\H^+}(w-\tilde w)=0$. Thanks to Proposition \ref{Unicite}, we obtain $w=\tilde w$, then $V=v$, which completes the proof of the decomposition theorem.
\end{proof}

The following proposition is a Poisson formula for the axisymmetric potentials in $\H^+$ :
\begin{prop}
 Let $m\in\C$ be such that Re $m<1$ and $u:\R\rightarrow\R$ be a continuous and bounded function.
Then there is a unique axisymmetric potential $U\in{C}^2(\H^+)$ such that $\lim_{\|(x,y)\|\to+\infty} U(x,y)=0$
and for all $y\in\R$,
$$
\lim_{(0,y)}U=u(y).
$$
Moreover, we have for all $(x,y)\in\H^+$,
$$
U(x,y)=C_m x^{1-m} \int_{\eta=-\infty}^\infty \frac{u(\eta)\, d\eta}{ (x^2+(y-\eta)^2)^{1-\frac{m}{2}}}
$$
where $C_m=\frac{1-m}{ 2\pi}\int_{\theta=0}^\pi \sin^{1-m}\theta\, d\theta=\frac{1}{ 2^m \pi}\frac{\Gamma^2\left(1-\frac{m}{2}\right)}{\Gamma(1-m)}$.
\end{prop}

\begin{proof}
We define $f(x,y)={x^{1-m}\over (x^2+(y-\eta)^2)^{1-{m\over2}}}$. To show that $U$ is a solution of $L_m U=0$, it suffices to prove that $L_m f=0$ by differentiation under the integral sign. We have
$$
\d_x f={(1-m)x^{-m}\over (x^2+(y-\eta)^2)^{1-{m\over2}}}-{(2-m)x^{2-m} \over (x^2+(y-\eta)^2)^{2-{m\over2}}}
$$
and
$$
\d_{xx} f=-{m(1-m)x^{-m-1}\over (x^2+(y-\eta)^2)^{1-{m\over2}}}-{(2-m)(3-2m)x^{1-m}\over(x^2+(y-\eta)^2)^{2-{m\over2}}}+{(2-m)(4-m)x^{3-m}\over (x^2+(y-\eta)^2)^{3-{m\over2}}}
$$
and
$$
\d_{yy}f=-{(2-m)x^{1-m}\over(x^2+(y-\eta)^2)^{2-{m\over2}}}+{(2-m)(4-m)(y-\eta)^2x^{1-m}\over (x^2+(y-\eta)^2)^{3-{m\over2}}}.
$$
Then,
$$
\Delta f={m(2-m)x^{1-m}\over(x^2+(y-\eta)^2)^{2-{m\over2}}}-{m(1-m)x^{-m-1}\over (x^2+(y-\eta)^2)^{1-{m\over2}}}
$$
and we deduce that $L_m f(x,y)=0$.

We have
$$
U(x,y)=C_m x^{1-m} \int_{\eta=-\infty}^\infty {u(\eta)\, d\eta\over (x^2+(y-\eta)^2)^{1-{m\over2}}}={C_m\over x}\int_{\eta=-\infty}^\infty {u(\eta)\, d\eta\over (1+({y-\eta\over x})^2)^{1-{m\over2}}}
$$
By a change of variable $t={y-\eta\over x}$, we obtain
$$
U(x,y)=C_m\int_{t=-\infty}^\infty {u(y-tx)\, dt\over (1+t^2)^{1-{m\over2}}}
$$
Thanks to the dominated convergence theorem, it suffices to show that
$$
C_m\int_{t=-\infty}^\infty  { dt\over (1+t^2)^{1-{m\over2}}}={1-m\over 2\pi}\int_{\theta=0}^\pi \sin^{1-m}\theta\, d\theta\int_{t=-\infty}^\infty  { dt\over (1+t^2)^{1-{m\over2}}}=1.
$$
To see this, according \cite{ablowitz} (page 258), note that
$$
\int_{t=-\infty}^\infty  { dt\over (1+t^2)^{1-{m\over2}}}=B\left({1\over2},{1-m\over2}\right)={\Gamma(1/2)\Gamma\left({1-m\over2}\right)\over \Gamma\left(1-{m\over2}\right)}
$$
where $B$ is the Euler beta function and
$$
{1-m\over 2\pi}\int_{\theta=0}^\pi \sin^{1-m}\theta\, d\theta={1-m\over 2\pi}2^{1-m}B\left(1-{m\over2},1-{m\over2}\right)={1-m\over 2\pi}2^{1-m}{\Gamma^2\left(1-{m\over2}\right)\over \Gamma\left(2-m\right)}.
$$
Then, using the duplication formula for the $\Gamma$ function,
$$
\Gamma(2z)=\pi^{-1/2}2^{2z-1}\Gamma(z)\Gamma\left(z+{1\over2}\right)
$$
and the recurrence formula $\Gamma(z+1)=z\Gamma(z)$, we obtain the desired result, 
$$
{\Gamma(1/2)\Gamma\left({1-m\over2}\right)\over \Gamma\left(1-{m\over2}\right)}{1-m\over 2\pi}2^{1-m}{\Gamma^2\left(1-{m\over2}\right)\over \Gamma\left(2-m\right)}=1.
$$
The uniqueness follows from the proposition \ref{Unicite}.
So, we proved the proposition.
\end{proof}

\begin{rem}
We could ask ourselves the question of the existence of a such reproducing formula if $\Re m\geq 1$. In fact, if $m\in\N^*$ and if $u\in C^2(\overline{\H^+})$ satisfies $L_m(u)=0$ on $\H^+$, then the function $v$ defined on $\R^{m+2}$ by $$v(x_1,\ldots,x_{m+2})=u(0,x_{m+2})$$ and $$v(x_1,\ldots,x_{m+2})=u\left(\sqrt{x_1^2+\cdots+x_{m+1}^2},x_{m+2}\right)$$ is harmonic on $(\R^{m+1})_*\times\R$. In particular, if $m\geq 2$, the Proposition 18 in \cite{dautray}, page 310 shows that $v$ can be extended to a harmonic function on $\R^{m+2}$, which tends to 0 at infinity. We then deduce that the function $v$ is identically zero, then $u\equiv0$, demonstrating that the problem to find a solution of $L_m(u)=0$ with $u$ vanishing at infinity and that the values of $u$ are known on the $y$-axis is a problem that does not make sense. In this case, no solution to the Dirichlet problem is a consequence of loss of the ellipticity of the equation $L_m u=0$ at the boundary of $\H^+$. Therefore,
we do not deal with the case $\Re m\geq 1$.
\end{rem}

\section{Fourier-Legendre decomposition}

First, we will introduce a specific system of coordinates named bipolar coordinates $(\tau,\theta)$ (see \cite{lebedev}) and numerical applications on extremal bounded problems using this system of coordinates have been realized in \cite{fischer3,fischer2,fischer1}.

 Let $\alpha>0$. We suppose that there is a positive charge at $A=(-\alpha,0)$ and a negative charge at $B=(\alpha,0)$ (the absolute values of the two charges are identical). The potential generated by this charges at a point $M$ is $\ln\left(MA\over MB\right)$ (modulo a multiplicative constant). 

\hfil
\resizebox{5cm}{!}{
\begin{tikzpicture}

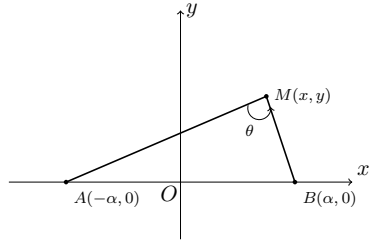
\draw (0.5,-1.5) node{Figure : Bipolar coordinates};

\draw[->](-3,0)--(3,0);
\draw[->](0,-1)--(0,3);

\draw[-,thick](-2,0)--(1.5,1.5);
\draw[-,thick](2,0)--(1.5,1.5);

\draw[<-] (1.575,1.31) arc (5:-195:2mm);

\draw (3.2,0.2) node {$x$};
\draw (0.2,3) node {$y$};
\draw (-0.2,-0.2) node {$O$};
\draw (2,-0.3) node[right] {\scriptsize $B(\alpha,0)$};
\draw (-2,-0.3) node[right] {\scriptsize $A(-\alpha,0)$};
\draw (1.5,1.5) node[right] {\scriptsize $M(x,y)$};
\draw (1,0.9) node[right] {\scriptsize $\theta$};

\fill[black]  (2,0) circle (0.04);
\fill[black]  (-2,0) circle (0.04);
\fill[black]  (1.5,1.5) circle (0.04);

\end{tikzpicture}}
\hfill

By definition, the bipolar coordinates are
$$
\tau:=\ln{MA\over MB}\quad {\rm and}\quad
\theta=\widehat{AMB}.
$$

The bipolar coordinates are linked to the Cartesian coordinates by the following formulas :
$$
x={\alpha\,\sh\tau\over\ch\tau-\cos\theta},\qquad y={\alpha\sin\theta\over\ch\tau-\cos\theta}.
$$
Let $R>0$ and $a=\sqrt{R^2+\alpha^2}$, the disk of center $(a,0)$ and of radius $R$ is defined in terms of bipolar coordinates by
$$
\tau\geq \tau_0=\ln\left({a\over R}+\sqrt{{a^2\over R^2}-1}\right)= {\rm argch}\,{ a\over R}.
$$
Moreover, the right half-plane is
$$
\H^+=\{(\tau,\theta)\,:\,\tau\in]0+\infty],\ \theta\in\left[0,2\pi\right[\}.
$$
The level lines $\tau=\tau_0$ are circles of center$(\alpha\coth\tau_0,0)$ and radii $\alpha/\sh\tau_0$. This implies that for all $\tau_0,\tau_1$ such that $0<\tau_0<\tau_1$, the set $\{(\tau,\theta),\, \tau\geq \tau_0\}$ is a closed disk and the set $\{(\tau,\theta),\,0<\tau<\tau_1\}$ is the complement on $\H^+$ of the closed disk $\{\tau\geq\tau_1\}$.

\hfil
\resizebox{10cm}{!}{
\begin{tikzpicture}

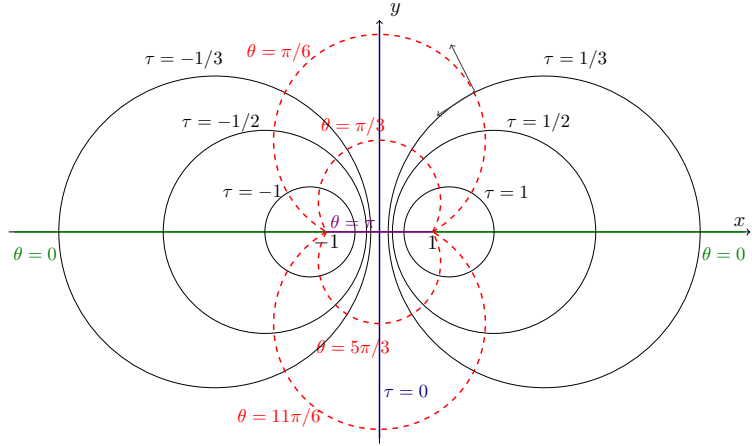
\draw (0,-4.5) node{Figure : Level lines (with $\alpha=1$)};

\draw[->] (-7,0) -- (7,0);
\draw[->] (0,-4) -- (0,4);
\draw (6.8,0.2) node {$x$};
\draw (0.3,4.2) node {$y$};

\draw (1,-0.2) node {$1$};
\draw (-1,-0.2) node {$-1$};

\draw (3,1.75) node[sloped,above]{ {\small $ \tau=1/2$}};
\draw (2.4,0.5) node[sloped,above] {{\small $ \tau=1$}};
\draw (3.7,2.95) node[sloped,above] {{\small $ \tau=1/3$}};
\draw (-3.7,2.95) node[sloped,above] {{\small $ \tau=-1/3$}};

\draw (-1.9,3.4) node[red] {{\small $ \theta=\pi/6$}};
\draw (-0.5,2) node[red] {{\small $ \theta=\pi/3$}};

\draw (-1.9,-3.5) node[red] {{\small $ \theta=11\pi/6$}};
\draw (-0.5,-2.2) node[red] {{\small $ \theta=5\pi/3$}};

\draw (-3,1.75) node[sloped,above]{{\small $ \tau=-1/2$}};
\draw (-2.4,0.5) node[sloped,above] {{\small $ \tau=-1$}};

\draw[-,green!50!black,thick] (1,0) -- (6.9,0);
\draw[-,green!50!black,thick] (-6.9,0) -- (-1,0);
\draw[-,blue!50!black,thick] (0,-3.9) -- (0,3.9);
\draw (6.5,-0.4) node[green!50!black] {{\small $\theta=0$}};
\draw (-6.5,-0.4) node[green!50!black] {{\small $\theta=0$}};
\draw (0.5,-3) node[blue!50!black] {{\small $\tau=0$}};

\draw (2.1639,0) circle (1.9190);
\draw (1.313,0) circle (0.8509);
\draw (-1.313,0) circle (0.8509);
\draw (-2.1639,0) circle (1.9190);
\draw (3.110,0) circle (2.945);
\draw (-3.110,0) circle (2.945);

\draw[red,dashed,thick] (0,1.732)+(-60:2) arc (-60:240:2);
\draw[red,dashed,thick] (0,-1.732)+(-240:2) arc (-240:60:2);
\draw[red,dashed,thick] (0,0.577)+(-30:1.1547) arc (-30:210:1.1547);
\draw[red,dashed,thick] (0,-0.577)+(-210:1.1547) arc (-210:30:1.1547);

\draw[->,darkgray] (1.8,2.65) -- (1.09,2.2) ;
\draw[->,darkgray] (1.8,2.65) -- (1.35,3.54) ;
\draw[-,violet,thick] (-1,0) -- (1,0);
\draw (-0.5,0.2) node[violet] {{\small $\theta=\pi$}};


\end{tikzpicture}}
\hfill

The following theorem is known for $m=-1$ by physicists (\cite{alladio1986,milligen1994,segura2000,shushkevich1997,cohl2000,love1972}. 
We extend this result to complex values of $m$ :

\begin{thm}\label{bipo} Let $u$ be a solution of $L_mu=0$ in an open set of $\H^+$ and putting
$$
v_m(\tau,\theta)=\sh^{{m-1\over 2}}\tau(\ch\tau-\cos\theta)^{-m/2}u(\tau,\theta)
$$
where by definition,
$$
\sh^{{m-1\over 2}}\tau(\ch\tau-\cos\theta)^{-m/2}=\exp\left({m-1\over2}\ln\sh\tau-{m\over2}\ln(\ch\tau-\cos\theta)\right)
$$
then
$$
{\d^2v_m\over\d\tau^2}+{\d^2v_m\over\d\theta^2}+\coth\tau{\d v_m\over\d\tau}+\left({1\over4}-{(m-1)^2\over4\,\sh^2\tau}\right)v_m=0.
$$
\end{thm}

\begin{proof}
We have
$$
{\d u\over\d\tau}=\alpha\left[{1-\ch\tau\cos\theta\over(\ch\tau-\cos\theta)^2}{\d u\over\d x}-{\sh\tau\sin\theta\over(\ch\tau-\cos\theta)^2}{\d u\over\d y}\right]
$$
and 
$$
{\d u\over\d\theta}=\alpha\left[{-\sh\tau\sin\theta\over(\ch\tau-\cos\theta)^2}{\d u\over\d x}+{\ch\tau\cos\theta-1\over(\ch\tau-\cos\theta)^2}{\d u\over\d y}\right].
$$
Thus, we obtain
$$
{\d u\over\d x}={1\over\alpha}\left((1-\ch\tau\cos\theta){\d u\over\d\tau}-\sh\tau\sin\theta{\d u\over\d\theta}\right),
$$

and
$$
{\d^2 u\over\d\tau^2}={\alpha^2\over (\ch\tau-\cos\theta)^4}\left[ (1-\ch\tau\cos\theta)^2{\d^2 u\over\d x^2}+\sh^2\tau\sin^2\theta {\d^2 u\over\d y^2} \qquad\qquad\qquad\qquad   \right.
$$
$$
\left.-2(1-\ch\tau\cos\theta)\sh\tau\sin\theta {\d^2 u\over\d x\d y} \right]$$
$$
+{\alpha\over(\ch\tau-\cos\theta)^3}\left[\sh\tau(\cos^2\theta+\ch\tau\cos\theta-2){\d u\over\d x}+{\sin\theta}\left(\ch^2\tau-2+\cos\theta\ch\tau\right){\d u\over\d y}\right]
$$
and
$$
{\d^2 u\over\d\theta^2}={\alpha^2\over (\ch\tau-\cos\theta)^4}\left[ \sh^2\tau\sin^2\theta {\d^2 u\over\d x^2}+(\ch\tau\cos\theta-1){\d^2 u\over\d y^2} \qquad\qquad\qquad\qquad   \right.
$$
$$
\left.
+2(1-\ch\tau\cos\theta)\sh\tau\sin\theta {\d^2 u\over\d x\d y}\right] 
$$
$$
+{\alpha\over(\ch\tau-\cos\theta)^3}\left[\sh\tau(2-\cos^2\theta-\cos\theta\ch\tau){\d u\over\d x}+\sin\theta(2-\ch^2\tau-\ch\tau\cos\theta){\d u\over\d y}\right]
$$
In particular, we have
$$
{\d^2 u\over\d\tau^2}+{\d^2 u\over\d\theta^2}={\alpha^2\over (\ch\tau-\cos\theta)^2}\left[{\d^2 u\over\d x^2}+{\d^2 u\over\d y^2}\right] .$$

Therefore, we obtain
$$
L_{m,x,y}u=\left({\ch\tau-\cos\theta\over\alpha}\right)^2\left({\d^2 u\over\d\tau^2}+{\d^2 u\over\d\theta^2}+{m(1-\ch\tau\cos\theta)\over\sh\tau(\ch\tau-\cos\theta)}{\d u\over\d\tau}-{m\sin\theta\over \ch\tau-\cos\theta}{\d u\over\d\theta}\right).
$$
We put
$$
u(\tau,\theta)={(\ch\tau-\cos\theta)^{m/2}\over\sh^{{m-1\over2}}\tau} v_m(\tau,\theta)
$$
and we calculate $L_{m,x,y}u$ in terms of $F(\tau,\theta)$. Denoting
$$
r_m(\tau,\theta)={(\ch\tau-\cos\theta)^{m/2}\over\sh^{m-1\over2}\tau},
$$ 
we have
$$
{\d r_m\over\d\theta}={m\over 2}{\sin\theta\over\ch\tau-\cos\theta}r_m
$$
and
$$
{\d^2 r_m\over\d\theta^2}= {m\over4(\ch\tau-\cos\theta)^2}\left(2\cos\theta\ch\tau+m\sin^2\theta-2\right)r_m
$$
then
$$
{\d r_m\over\d\tau}={1\over(\ch\tau-\cos\theta)\sh\tau}\left(\ch^2\tau+(m-1)\ch\tau\cos\theta-m\right)r_m
$$
and
$$
{\d^2 r_m\over\d\tau^2}={1\over4(\ch\tau-\cos\theta)^2\sh^2\tau}\left[\ch^4\tau-2\ch^3\tau\cos\theta+(m-1)^2\ch^2\tau\cos^2\theta+
\right.\qquad\qquad\qquad
$$
$$
\qquad\qquad\left.+2(m-1)\ch^2\tau+(4-2m^2)\ch\tau\cos\theta+2(m-1)\cos^2\theta+m(m-2)\right]r_m.
$$

The equation
$$
L_{m,x,y}u=0
$$
can be rewritten as
$$
r_m\left({\d^2v_m\over\d\tau^2}+{\d^2v_m\over\d\theta^2}\right)+{\d v_m\over\d\tau}\left(2{\d r_m\over\d\tau}+{m\over\sh\tau}{1-\ch\tau\cos\theta\over\ch\tau-\cos\theta}r_m\right)+\qquad\qquad
$$
$$
+{\d v_m\over\d\theta}\left(2{\d r_m\over\d\theta}-{m\sin\theta\over\ch\tau-\cos\theta}r_m\right)
$$
$$
\qquad\qquad+v_m\left(
{\d^2 r_m\over\d\tau^2}+{\d^2 r_m\over\d\theta^2}+{m(1-\ch\tau\cos\theta)\over\sh\tau(\ch\tau-\cos\theta)}{\d r_m\over\d\tau}-{m\sin\theta\over \ch\tau-\cos\theta}{\d r_m\over\d\theta}\right)=0
$$
with
$$
2{\d r_m\over\d\tau}+{m\over\sh\tau}{1-\ch\tau\cos\theta\over\ch\tau-\cos\theta}r_m=r_m\coth\tau,
$$
$$
2{\d r_m\over\d\theta}-{m\sin\theta\over\ch\tau-\cos\theta}r_m=0
$$
and
$$
{\d^2 r_m\over\d\tau^2}+{\d^2 r_m\over\d\theta^2}+{m(1-\ch\tau\cos\theta)\over\sh\tau(\ch\tau-\cos\theta)}{\d r_m\over\d\tau}-{m\sin\theta\over \ch\tau-\cos\theta}{\d r_m\over\d\theta}=\left({1\over4}-{(m-1)^2\over4\sh^2\tau}\right)r_m.
$$
And this completes the proof.
\end{proof}
\bigskip

We seek $v_m$ by separation of variables : $v_m(\tau,\theta)=A_m(\tau)B_m(\theta)$. 
From the equation satified by $v_m$ (see Theorem \ref{bipo}), we obtain
$$
{A_m''\over A_m}+\coth\tau{A'_m\over A_m}+{1\over4} -{(m-1)^2\over4\,\sh^2\tau}=-{B_m''\over B_m}. 
$$
The term on the right depends only of $\theta$ and the left one depends only of $\tau$, thus we deduce that both are constant. Let $n\in\C$ such that this constant is equal to $n^2$. We then have
$$
\begin{cases}
\displaystyle{A_m''+\coth\tau A_m'+\left(\frac14 -\frac{(m-1)^2}{4\sh^2\tau}-n^2\right)A_m=0,}\cr
\displaystyle{B_m''+n^2B_m=0.}
\end{cases}
$$
$B_m$ is naturally a $2\pi-$periodic function (because $\theta$ represents an angle), therefore $n$ should necessarily be an integer.

To examine the equation satisfied by $A_m$, we carry out the following change of function 
$$
A_m(\tau)=C_m(\ch\tau).
$$
Then, $C_m$ satisfies
$$
\sh^2\tau C_m''(\ch\tau)+2\,\ch\tau\, C_m'(\ch\tau)+\left({1\over 4}-n^2-{(m-1)^2\over 4\sh^2\tau}\right)C_m(\ch\tau)=0
$$
which can be rewritten as
$$
(1-\ch^2\tau)C_m''(\ch\tau)-2\,\ch\tau\, C_m'(\ch\tau)+\left(n^2-{1\over4}-{((m-1)/2)^2\over 1-\ch^2\tau}\right)C_m(\ch\tau)=0.\eqno{(LAH)}
$$
This equation is named {\it  Hyperbolic Associated Legendre equation}.

Note that if we put $z=\ch\tau$ and $u(z)=C_m(\ch\tau)$, then
$$
(1-z^2)u''-2zu'+\left[\nu(\nu+1)-{\mu^2\over1-z^2}\right]u=0\eqno{(LA)}
$$
where
$$
\nu=n-{1\over2}\qquad\hbox{and}\qquad \mu={m-1\over2}.
$$
This equation is named {\it Associated Legendre equation}, and it can be reduced to the Legendre equation if  $\mu=0$ :
$$
(1-z^2)u''-2zu'+\nu(\nu+1)u=0.\eqno{(L)}
$$ 

\bigskip
Two independent solutions of {\it (LA)} are given in section~\ref{sec:annexe} and denoted $P_\nu^\mu(\ch\tau)$ and $Q_\nu^\mu(\ch\tau)$.

Starting from this investigation of solutions in the form of separate variables, we can state the following theorem

\begin{thm}\label{A}
Let $m\in\C$. Let $0<\tau_0$. Let $u$ be a smooth solution of $L_mu=0$ on the disk $\tau\geq\tau_0$ and let $v$ be a smooth solution of $L_mv=0$ on 
$\H^+\setminus\{\tau>\tau_0\}$ which is the complement on $\H^+$ of the disk $\{\tau>\tau_0\}$ and we assume that $\lim_{\d\H^+}v=0$. Then there are two sequences $(a_n)_{n\in\Z}$ and 
$(b_n)_{n\in\Z}$ of $\ell^2(\Z)$ (which are even rapidly decreasing) such that :
$$
u=\sum_{n=-\infty}^{+\infty} a_n Q_{n-{1\over2}}^{m-1\over2}(\ch\tau)\sh^{1-m\over2}\tau(\ch\tau-\cos\theta)^{m\over2}e^{in\theta}
$$
and
$$
v=\sum_{n=-\infty}^{+\infty} b_n P_{n-{1\over2}}^{m-1\over2}(\ch\tau)\sh^{1-m\over2}\tau(\ch\tau-\cos\theta)^{m\over2}e^{in\theta}.
$$
The sequence $(a_n)$ is unique.
In addition, the convergence of the first series is uniform on every compact set $[\tau_1,\tau_2]$ of the disk $\tau>\tau_0$  with $\tau_0\leq\tau_1<\tau_2$.
And the convergence of the second one is uniform on every compact set $[\tau_3,\tau_4]$ of the complement of the disk $\tau>\tau_0$ on $\H^+$ with $0<\tau_3<\tau_4\leq\tau_0$.

If Re $m<1$, then the sequence $(b_n)$ is unique.
\end{thm}

\begin{proof}
Indeed, decomposing the function
$$
\theta\mapsto u(\tau_0,\theta)(\ch\tau_0-\cos\theta)^{-m/2}\sh^{m-1\over2}\tau_0
$$
by Fourier series with respect to the variable $\theta$, to yield the Fourier expansion for $u(\tau_0,\cdot)$
$$
u(\tau_0,\theta)=\sh^{1-m\over2}\tau_0(\ch\tau_0-\cos\theta)^{m\over2}\sum_{n=-\infty}^{+\infty}a_n e^{in\theta},
$$
where $a_n\in\ell^2(\Z)$ satisfies
$$
a_n={1\over 2\pi}\int_0^{2\pi}(\ch\tau_0-\cos \theta)^{-m/2}\sh^{m-1\over2}\tau_0\ u(\tau_0,s)e^{-ins}\,ds.
$$
This function is a smooth function of the variable $\theta$, therefore we deduce that the sequence $(a_n)_n$ is rapidly decreasing when $|n|\to+\infty$.
The function
$$\tilde{u}(\tau,\theta)=\sh^{1-m\over2}\tau(\ch\tau-\cos\theta)^{m\over2}\sum_{n=-\infty}^{+\infty}a_n {Q_{n-{1\over2}}^{m-1\over2}(\ch\tau)\over Q_{n-{1\over2}}^{m-1\over2}(\ch\tau_0)}e^{in\theta}$$
coincides  with $u$ on the circle $\tau=\tau_0$.

Moreover, thanks to the Proposition \ref{estimationPQ}, we have when $|n|\to+\infty$,
$$
{Q_{n-{1\over2}}^{m-1\over2}(\ch\tau)\over Q_{n-{1\over2}}^{m-1\over2}(\ch\tau_0)}\sim\sqrt{\sh\tau_0\over\sh\tau}e^{|n|(\tau_0-\tau)}
$$
and this equivalence is uniform in all compact set $[\tau_1,\tau_2]$ with $0<\tau_0\leq \tau_1<\tau_2$.

It follows that the series of functions which defines $\tilde u$ is normally converging on any compact sets $[\tau_1,\tau_2]$ of the disk $\tau\geq\tau_0$. So does same for derivatives with respect to $\tau$ and $\theta$ (which are expressed also with the Associated Legendre functions as mentioned in the section \ref{sec:annexe}).

Particularly, the function $\tilde u$ is well defined on the disk $\tau\geq\tau_0$ and coincides with $u$ on the circle $\tau=\tau_0$. 

Due to the fact that the solution of an elliptic equation is uniquely determined  by its boundary values ​​(this follows from the maximum principle), we deduce that $\tilde u$ the unique axisymmetric potential on the disk $\tau\geq\tau_0$ which coincides with $u$ on the circle $\tau=\tau_0$.

For $v$, the proof is completely similar.

Indeed, decomposing the function
$$
\theta\mapsto v(\tau_0,\theta)(\ch\tau_0-\cos\theta)^{-m/2}\sh^{m-1\over2}\tau_0
$$
by Fourier series with respect to the variable $\theta$, to yield the Fourier expansion for $v(\tau_0,\cdot)$
$$
v(\tau_0,\theta)=\sh^{1-m\over2}\tau_0(\ch\tau_0-\cos\theta)^{m\over2}\sum_{n=-\infty}^{+\infty}b_n e^{in\theta},
$$
where $b_n\in\ell^2(\Z)$ satisfies
$$
b_n={1\over 2\pi}\int_0^{2\pi}(\ch\tau_0-\cos \theta)^{-m/2}\sh^{m-1\over2}\tau_0\ v(\tau_0,s)e^{-ins}\,ds.
$$
This function is a smooth function of the variable $\theta$, therefore we deduce that the sequence $(b_n)_n$ is rapidly decreasing when $|n|\to+\infty$.
The function
$$\tilde{v}(\tau,\theta)=\sh^{1-m\over2}\tau(\ch\tau-\cos\theta)^{m\over2}\sum_{n=-\infty}^{+\infty}b_n {P_{n-{1\over2}}^{m-1\over2}(\ch\tau)\over P_{n-{1\over2}}^{m-1\over2}(\ch\tau_0)}e^{in\theta}$$
coincides with $v$ on the circle $\tau=\tau_0$.

Moreover, thanks to the Proposition \ref{estimationPQ}, we have when $|n|\to+\infty$,
$$
{P_{n-{1\over2}}^{m-1\over2}(\ch\tau)\over P_{n-{1\over2}}^{m-1\over2}(\ch\tau_0)}\sim\sqrt{\sh\tau_0\over\sh\tau}e^{|n|(\tau-\tau_0)}
$$
and this equivalence is uniform in all compact set $[\tau_1,\tau_2]$ with $0<\tau_1<\tau_2\leq\tau_0$.

It follows that the series of functions which defines $\tilde v$ is normally converging on any compact sets $[\tau_1,\tau_2]$ of the complementary of the disc $\tau>\tau_0$. So does same for derivatives with respect to $\tau$ and $\theta$.

Particularly, the function $\tilde v$ is well defined on the complementary of the disk $\tau>\tau_0$ and coincides with $v$ on the circle $\tau=\tau_0$. 

We will show that
$$
\lim_{\tau\to0+}\tilde v=0.
$$
If Re $m<1$, we have when $n\in\N$ and thanks to the formula (\ref{P2})
$$
P_{n-{1\over2}}^{m-1\over2}(\ch\tau)={2^{m-1\over2}\over\sqrt\pi\Gamma\left(1-{m\over2}\right)}\sh^{1-m\over2}\tau\int_0^\pi(\ch\tau+\sh\tau\,\cos\theta)^{n+{m\over2}-1}\sin^{1-m}\theta\,d\theta
$$
then
$$
\lim_{\tau\to0+}P_{n-{1\over2}}^{m-1\over2}(\ch\tau)=0
$$
and in addition, for $n>1-{{\rm Re\,}m\over2}$, we have
$$
\left|P_{n-{1\over2}}^{m-1\over2}(\ch\tau)\right|\leq
{2^{{\rm Re\,}m-1\over2}\sh^{1-{\rm Re\,}m\over2}\tau\over\sqrt\pi\left|\Gamma\left(1-{m\over2}\right)\right|}\int_0^\pi\left(\ch\tau+\sh\tau\cos\theta)\right)^{n+{{\rm Re\,m}\over2}-1}\sin^{1-{\rm Re\,}m}\theta\,d\theta
$$
$$
\leq{2^{{\rm Re\,}m-1\over2}\sh^{1-{\rm Re\,}m\over2}\tau\over\sqrt\pi\left|\Gamma\left(1-{m\over2}\right)\right|}\int_0^\pi\left(\ch\tau+\sh\tau)\right)^{n+{{\rm Re\,m}\over2}-1}\sin^{1-{\rm Re\,}m}\theta\,d\theta
\leq
C_m\sh^{1-{\rm Re\,}m\over2}\tau e^{\left(n+{{\rm Re}\,m\over2}\right)\tau}
$$
thus
$$
\sum_{n>1-{{\rm Re}\,m\over2}}\sup_{\tau\in\left[0,{\tau_0\over2}\right]}\left| b_n{P_{n-{1\over2}}^{m-1\over2}(\ch\tau)\over P_{n-{1\over2}}^{m-1\over2}(\ch\tau_0)}e^{in\theta}\right|<+\infty
$$
by the Proposition \ref{estimationPQ}, we obtain
$$
P_{n-{1\over2}}^{m-1\over2}(\ch\tau_0)\sim_{n\to+\infty}{n^{{m\over2}-1}\over\sqrt{2\pi\,\sh\tau_0}}e^{n\tau_0}.
$$
So, we can deduce that $\lim_{\tau\to0+}\tilde v=0$.

It remains to prove the uniqueness of the previous decomposition where Re$\,m<1$. This will result in the next paragraph which will establish the fact that the family
$$
\mathcal{A}:=\left(\frac{Q_{n-\frac12}^{m-1\over2}(\ch\tau)}{Q_{n-\frac12}^{m-1\over2}(\ch\tau_0)}{(\ch\tau-\cos\theta)^{m/2}\over\sh^{m-1\over2}\tau}e^{in\theta}\right)_{n\in\Z}:=(a_n)_{n\in\Z}
$$
$$
\mathcal{B}:=\left(\frac{P_{n-\frac12}^{m-1\over2}(\ch\tau)}{P_{n-\frac12}^{m-1\over2}(\ch\tau_1)}{(\ch\tau-\cos\theta)^{m/2}\over\sh^{m-1\over2}\tau}e^{in\theta}\right)_{n\in\Z}:=(b_n)_{n\in\Z}$$\\
is a Riesz basis.
\end{proof}

\begin{cor}
The solution of the Dirichlet problem for $L_mu=0$ on $D((a,0),R)$ where $u=\varphi$ on $\d D((a,0),R)$ is given by
$$
u(\tau,\theta)=\sh^{1-m\over2}\tau(\ch\tau-\cos\theta)^{m\over2}\sum_{n=-\infty}^{+\infty}c_n {Q_{n-{1\over2}}^{m-1\over2}(\ch\tau)\over Q_{n-{1\over2}}^{m-1\over2}(\ch\tau_0)}e^{in\theta}
$$
where $\{\tau=\tau_0\}$ corresponds to the circle of center $(a,0)$ and radius $R$ and where
$$
c_n={1\over 2\pi}\int_0^{2\pi}(\ch\tau_0-\cos \theta)^{-m/2}\sh^{m-1\over2}\tau_0\ \varphi(a+R\cos s, R\sin s)e^{-ins}\,ds.
$$
Similarly,
$$
v(\tau,\theta)=\sh^{1-m\over2}\tau(\ch\tau-\cos\theta)^{m\over2}\sum_{n=-\infty}^{+\infty}c_n {P_{n-{1\over2}}^{m-1\over2}(\ch\tau)\over P_{n-{1\over2}}^{m-1\over2}(\ch\tau_0)}e^{in\theta}
$$
is a solution of $L_mv=0$ on $\H^+\setminus D((a,0),R)$, which is equal to $\varphi$ on $\d D((a,0),R)$ where
$$
c_n={1\over 2\pi}\int_0^{2\pi}(\ch\tau_0-\cos \theta)^{-m/2}\sh^{m-1\over2}\tau_0\ \varphi(a+R\cos s, R\sin s)e^{-ins}\,ds.
$$
Moreover, if Re$\,m<1$, then $v$ satisfies $\lim_{\d\H^+}v=0$, and the function $v$ constructed above is the unique solution of the Dirichlet problem
$L_mv=0$ on $\H^+\setminus D((a,0),R)$ which vanishes on $\d\H^+$.
\end{cor}

\section{Riesz basis}

We will prove that the half part of the following family

$$
\left(
{(\ch\tau-\cos\theta)^{m/2}\over\sh^{{m-1\over2}}\tau}
\left\{
\begin{array}{c}
\cos(n\theta)\cr
\sin(n\theta) 
\end{array}\right\}
\left\{
\begin{array}{c}
P_{n-\frac12}^{{m-1\over2}}(\ch\tau)\cr
Q_{n-\frac12}^{{m-1\over2}}(\ch\tau)
\end{array}\right\}\right)_{n\in\Z}
$$
is a basis of solutions on the disk $\tau\geq\tau_1$ and the other half part is a basis of solutions on $\tau\leq\tau_0$, which is the complement on $\H^+$ of a some disk, with $0<\tau_0<\tau_1$. This fact is known for $m=-1$, namely for $\mu=1$. We extend this result for complex values of $m$.

\bigskip

Let us recall the definition of a Riesz basis. $(x_n)_{n\in\N}$ is  a quasi-orthogonal or Riesz sequence of a Hilbert space $X$ if there are two constants $c,\,C>0$ such that for all sequences $(a_n)_{n\in\Z}$ in $\ell^2$, we have
$$
c^2\sum_n|a_n|^2\leq\left\|\sum_n a_n x_n\right\|^2 \leq C^2\sum_n|a_n|^2.
$$
If the family $(x_n)_{n\in\Z}$ is complete, it is a {\it Riesz basis}.
The matrix of scalar product $\{\langle x_i,x_j\rangle\}_{i,j}$ is named Gram matrix associated to $\{x_i\}_i$.
\bigskip

To prove that $\{x_i\}_i$ is a Riesz basis, a convenient characterization with the Gram matrix is the following property :

\bigskip

 {\bf Property (\cite[\text{p. }170]{nikolski}).} {\it A family $\{x_i\}_i$ is a Riesz basis for a some Hilbert space if $\{x_i\}_i$ is complete on this Hilbert space and if the Gram matrix associated to  $\{x_i\}_i$ defines a bounded and invertible operator on $\ell^2(\N)$.\\
}

Let $\mathcal{A}$ and $\mathcal{B}$ the two families of solutions of $L_m[u]=0$, respectively inside the disk $\tau>\tau_0$ and outside the other one $\tau>\tau_1$, with $0<\tau_0<\tau_1$
$$
\mathcal{A}:=\left(\frac{Q_{n-\frac12}^{m-1\over2}(\ch\tau)}{Q_{n-\frac12}^{m-1\over2}(\ch\tau_0)}{(\ch\tau-\cos\theta)^{m/2}\over\sh^{m-1\over2}\tau}e^{in\theta}\right)_{n\in\Z}:=(a_n)_{n\in\Z}
$$
$$
\mathcal{B}:=\left(\frac{P_{n-\frac12}^{m-1\over2}(\ch\tau)}{P_{n-\frac12}^{m-1\over2}(\ch\tau_1)}{(\ch\tau-\cos\theta)^{m/2}\over\sh^{m-1\over2}\tau}e^{in\theta}\right)_{n\in\Z}:=(b_n)_{n\in\Z}$$\\
Let $\mathcal{C}$ the union of the two previous families :
$$
\mathcal{C}:=\left(c_n\right)_{n\in\Z}:=\left(c_{2n}=a_{n}\text{ et }c_{2n+1}=b_{n}\right)_{n\in\Z}
$$
The annulus defined in terms of bipolar coordinates $\left\{0<\tau_0<\tau<\tau_1\right\}$ will be denoted $\mathbb{A}$. 
The space $L^2(\d\A)$ is equipped of the following inner product~: for $f,g\in L^2(\d\A)$,
$$
\langle f,g\rangle=\frac{1}{2\pi}\int_{0}^{2\pi}f(\tau_0,\theta)\overline{g(\tau_0,\theta)}\frac{\sh^{\Re m-1}\tau_0}{(\ch\tau_0-\cos\theta)^{\Re m}}\text{d}\theta
\qquad\qquad\qquad\qquad\qquad\qquad\qquad\qquad$$
$$
\qquad\qquad\qquad\qquad\qquad\qquad+\frac{1}{2\pi}\int_{0}^{2\pi}f(\tau_1,\theta)\overline{g(\tau_1,\theta)}\frac{\sh^{\Re m-1}\tau_1}{(\ch\tau_1-\cos\theta)^{\Re m}}\text{d}\theta.
$$

We have the following proposition :

\bigskip

\begin{prop}
$\mathcal{C}$ is a Riesz basis in the Hilbert space $L^2(\d\mathbb{A})$.
\end{prop}

\begin{proof}
Indeed, in order to build the Gram matrix of $\mathcal{C}$, we first calculate all its scalar products. We obtain for all $n\in\Z$,
\begin{align*}
\langle c_{2n},c_{2n}\rangle&=1+\left|\frac{Q_{n-\frac12}^{m-1\over2}(\ch\tau_1)}{Q_{n-\frac12}^{m-1\over2}(\ch\tau_0)}\right|^2&\\
\langle c_{2n+1},c_{2n+1}\rangle&=1+\left|\frac{P_{n-\frac12}^{m-1\over2}(\ch\tau_0)}{P_{n-\frac12}^{m-1\over2}(\ch\tau_1)}\right|^2&\\
\langle c_{2n},c_{2n+1}\rangle&=\overline{\left(\frac{P_{n-\frac12}^{m-1\over2}(\ch\tau_0)}{P_{n-\frac12}^{m-1\over2}(\ch\tau_1)}\right)}+\frac{Q_{n-\frac12}^{m-1\over2}(\ch\tau_1)}{Q_{n-\frac12}^{m-1\over2}(\ch\tau_0)}&\\
\langle c_{2n+1},c_{2n}\rangle&=\overline{\langle c_{2n},c_{2n+1}\rangle}&
\end{align*}
In all other cases, the inner product is zero, the Gram matrix is diagonal by blocks  and each blocks is expressed as the $2\times 2$ matrix :
$$
\displaystyle
M_n=
\left(\begin{array}{clrr} 
\displaystyle1+\left|\frac{Q_{n-\frac12}^{m-1\over2}(\ch\tau_1)}{Q_{n-\frac12}^{m-1\over2}(\ch\tau_0)}\right|^2& & \displaystyle \overline{\left(\frac{P_{n-\frac12}^{m-1\over2}(\ch\tau_0)}{P_{n-\frac12}^{m-1\over2}(\ch\tau_1)}\right)}+\frac{Q_{n-\frac12}^{m-1\over2}(\ch\tau_1)}{Q_{n-\frac12}^{m-1\over2}(\ch\tau_0)}\cr
\cr
\displaystyle\frac{P_{n-\frac12}^{m-1\over2}(\ch\tau_0)}{P_{n-\frac12}^{m-1\over2}(\ch\tau_1)}+\overline{\left(\frac{Q_{n-\frac12}^{m-1\over2}(\ch\tau_1)}{Q_{n-\frac12}^{m-1\over2}(\ch\tau_0)}\right)}& & 1+\displaystyle\left|\frac{P_{n-\frac12}^{m-1\over2}(\ch\tau_0)}{P_{n-\frac12}^{m-1\over2}(\ch\tau_1)}\right|^2
\end{array}\right)
$$
The Gram matrix $G$ can be written as
$$
\displaystyle
G=
\left(\begin{array}{lllccrrr} 
M_0        & 0      & \cdots &\cdots &\cdots&\cdots&\cdots &\cdots \cr
0          & M_{-1} & 0      &\cdots &\cdots&\cdots&\cdots &\cdots \cr
\vdots     & 0      & M_1    & 0     &\cdots&\cdots&\cdots &\cdots \cr
\vdots     &\ddots  &  0     & M_{-2}& 0    &\cdots&\cdots &\cdots \cr
\vdots     &\ddots  & \ddots & 0     &\ddots&\ddots&\cdots &\cdots \cr
\vdots     &\ddots  & \ddots & \ddots&\ddots&M_{-n}&\ddots &\cdots \cr
\vdots     &\ddots  & \ddots & \ddots&\ddots&\ddots& M_n   &\ddots \cr
\vdots     &\ddots  & \ddots & \ddots&\ddots&\ddots&\ddots&\ddots 
\end{array}\right)
$$
and the determinant of $M_n$ is
$$
\det (M_n)=\left|1-\frac{Q_{n-\frac12}^{m-1\over2}(\ch\tau_1)}{Q_{n-\frac12}^{m-1\over2}(\ch\tau_0)}\frac{P_{n-\frac12}^{m-1\over2}(\ch\tau_0)}{P_{n-\frac12}^{m-1\over2}(\ch\tau_1)}\right|^2
$$

Let's show that  $M_n$ is invertible. Suppose the contrary, if $M_n$ is not invertible, then
$\det(M_n)=0$, which is equivalent to
$$
Q_{n-\frac12}^{m-1\over2}(\ch\tau_1)P_{n-\frac12}^{m-1\over2}(\ch\tau_0)=Q_{n-\frac12}^{m-1\over2}(\ch\tau_0)P_{n-\frac12}^{m-1\over2}(\ch\tau_1).
$$
The previous equality can be written as follows
$$
\left|\begin{array}{clrr}    
   Q_{n-\frac12}^{m-1\over2}(\ch\tau_0) & P_{n-\frac12}^{m-1\over2}(\ch\tau_0)  \\
   Q_{n-\frac12}^{m-1\over2}(\ch\tau_1) & P_{n-\frac12}^{m-1\over2}(\ch\tau_1)
   \end{array}\right|=0, \text{with   $P_{n-\frac12}^{m-1\over2}(\ch\tau_0)$, $Q_{n-\frac12}^{m-1\over2}(\ch\tau_1)\neq0$.}
$$  

Therefore, there is $\lambda\in\C\setminus\{0\}$ (which depends on $m$, $n$, $\tau_0$ and $\tau_1$) such that
$$
\left\{\begin{array}{c}
Q_{n-\frac12}^{m-1\over2}(\ch\tau_0)=\lambda P_{n-\frac12}^{m-1\over2}(\ch\tau_0)\cr
Q_{n-\frac12}^{m-1\over2}(\ch\tau_1)=\lambda P_{n-\frac12}^{m-1\over2}(\ch\tau_1)
\end{array}
\right.
$$
Then, by the asymptotic of Associated Legendre functions (see Proposition \ref{estimationPQ} in the Annex), on the one hand, we have both
$$
\lambda\mathop{\sim}_{n\to+\infty} \pi e^{i\pi {m-1\over2}}e^{-2n\tau_0}
$$
and on the other hand, we have
$$
\lambda\mathop{\sim}_{n\to+\infty}  \pi e^{i\pi {m-1\over2}}e^{-2n\tau_1},
$$
then it implies that $\tau_0=\tau_1$, it is not possible. 

We deduce that the matrix $M_n$ is invertible and this completes the proof.
\end{proof}

\section{Annex : Associated Legendre functions of first and second kind}
\label{sec:annexe}

In this section, we provide the main formulas of integral representation for the Associated Legendre function of the first and the second kind with $z=\ch\tau>1$ (see \cite{abramowitz,lebedev,virchenko}) :

$$
P_\nu^\mu(\ch\tau)={2^{-\nu}\sh^{-\mu}\tau\over\Gamma(-\mu-\nu)\Gamma(\nu+1)}\int_0^\infty(\ch\tau+\ch\theta)^{\mu-\nu-1}\sh^{2\nu+1}\theta\,d\theta
$$
with Re $\nu>-1$ and Re$(\mu+\nu)<0$.

\begin{equation}\label{P2}
P_\nu^\mu(\ch\tau)={2^\mu\sh^{-\mu}\tau\over\sqrt\pi\Gamma\left({1\over2}-\mu\right)}\int_0^\pi{(\ch\tau+\sh\tau\cos\theta)^{\mu+\nu}\over\sin^{2\mu}\theta}d\theta
\end{equation}
with Re $\mu<{1\over2}$.

$$
P_\nu^\mu(\ch\tau)=\sqrt{2\over\pi}{\sh^\mu\tau\over\Gamma\left({1\over2}-\mu\right)}\int_0^\tau{\ch\left[\left(\nu+{1\over2}\right)\theta\right]\over(\ch\tau-\ch\theta)^{\mu+1/2}}d\theta
$$
with Re $\mu<{1\over2}$.

$$
Q_\nu^\mu(\ch\tau)={e^{i\pi\mu}\sqrt\pi\over2^\mu}{\sh^\mu\tau\Gamma(\nu+\mu+1)\over\Gamma(\nu-\mu+1)\Gamma(\mu+1/2)}\int_0^\infty{\sh^{2\mu}\theta\over(\ch\tau+\sh\tau\ch\theta)^{\nu+\mu+1}}d\theta
$$
with Re $\mu>-{1\over2}$, Re$(\mu-\nu-1)<0$ and $\mu+\nu+1\notin\Z^-$.

$$
Q_\nu^\mu(\ch\tau)=\sqrt{\pi\over2}e^{i\pi\mu}{\sh^\mu\tau\over\Gamma\left({1\over2}-\mu\right)}\int_\tau^\infty{e^{-\left(\nu+{1\over2}\right)\theta}\over(\ch\theta-\ch\tau)^{\mu+1/2}}d\theta
$$
with Re $\mu<{1\over2}$ et Re$(\mu+\nu+1)>0$.

$$
Q_\nu^\mu(\ch\tau)=e^{i\pi\mu}2^{-\nu-1}{\Gamma(\nu+\mu+1)\over\Gamma(\nu+1)}\sh^{-\mu}\tau\int_0^\pi(\ch\tau+\cos\theta)^{\mu-\nu-1}\sin^{2\nu+1}\theta\,d\theta
$$
with Re $\nu>-1$ and $\mu+\nu+1\not\in\Z^-$ (see \cite{virchenko} pages 4, 5 and 6).

We have also the following relations satisfied by the Legendre functions  (see \cite{virchenko} page 6 and \cite{abramowitz}, formula 8.2.2)

$$
P_\nu^\mu=P_{-\nu-1}^\mu.
$$
$$
Q_{-\nu-1}^\mu(z)={-\pi e^{i\pi\mu}\cos(\pi\nu) P_\nu^\mu+\sin[\pi(\nu+\mu)]Q_\nu^\mu\over\sin[\pi(\nu-\mu)]}
$$
for $\nu-\mu\not\in\Z$.
(in particular for $\nu=n-{1\over2}$ with $n\in\Z$, we have
$$
Q_{-\nu-1}^\mu=Q_\nu^\mu
$$
for all $\mu\in\C$),
$$
e^{i\pi\mu}\Gamma(\nu+\mu+1)Q_\nu^{-\mu}=e^{-i\pi\mu}\Gamma(\nu-\mu+1)Q_\nu^\mu,
$$
$$
P_\nu^{-\mu}={\Gamma(\nu-\mu+1)\over\Gamma(\nu+\mu+1)}\left[P_\nu^\mu-{2\over\pi}e^{-i\pi\mu}\sin(\pi\mu)Q_\nu^\mu\right],
$$

In addition, we have the Whipple formulas connecting the associated Legendre functions of first and second kind (see \cite{virchenko} page 6)
$$
Q_\nu^\mu(\ch\tau)=e^{i\pi\mu}\sqrt{\pi\over2}{\Gamma(\mu+\nu+1)\over\sqrt{\sh\tau}}P_{-\mu-{1\over2}}^{-\nu-{1\over2}}(\coth\tau),
$$
$$
P_\nu^\mu(\ch\tau)={i e^{i\pi\nu}\over\Gamma(-\nu-\mu)}\sqrt{2\over\pi}{1\over\sqrt{\sh\tau}}Q_{-\mu-{1\over2}}^{-\nu-{1\over2}}(\coth\tau).
$$
We also have the recursion formulas  (see \cite{virchenko} pages 6 et 7)
$$
P_\nu^{\mu+1}(\ch\tau)={(\nu-\mu)\ch\tau\, P_\nu^\mu(\ch\tau)-(\nu+\mu)P_{\nu-1}^\mu(\ch\tau)\over\sh\tau}
$$
$$
(\nu-\mu+1)P_{\nu+1}^\mu(\ch\tau)=(2\nu+1)\ch\tau\, P_\nu^\mu(\ch\tau)-(\nu+\mu)P_{\nu-1}^\mu(\ch\tau).
$$
$$
(z^2-1){d P_\nu^\mu(z)\over dz}=(\nu+\mu)(\nu-\mu+1)(z^2-1)^{1/2}P_\nu^{\mu-1}(z)-\mu zP_\nu^\mu(z).
$$
$$
(z^2-1){d P_\nu^\mu(z)\over dz}=\nu z P_\nu^\mu(z)-(\nu+\mu) P_{\nu-1}^\mu(z).
$$
All of these formulas are used to explicitly calculate the values ​​of $P_\nu^\mu(\ch\tau)$ and $Q_\nu^\mu(\ch\tau)$ for all $\tau>0$ and $(\mu,\nu)\in\C^2$.

If $\mu$ and $\tau$ are fixed, the following proposition collects the behavior of Associated Legendre functions of the first and second kind when $\nu=n-{1\over2}$ with $n\in\Z$ and $|n|\to+\infty$.

\begin{prop}\label{estimationPQ}
We fix $\tau>0$ and $\mu\in\C$. Then if $\nu=n-{1\over2}$ with $n\in\Z$, we have :
$$
\hbox{when }\nu\to+\infty,\qquad P_\nu^\mu(\ch\tau)\sim {e^{\tau/2}\over\sqrt{2\pi\,\sh\tau}}\nu^{\mu-1/2}e^{\tau\nu}
$$
$$
\hbox{when }\nu\to-\infty,\qquad P_\nu^\mu(\ch\tau)\sim {e^{-\tau/2}\over\sqrt{2\pi\,\sh\tau}}(-\nu)^{\mu-1/2}e^{-\tau\nu}
$$
$$
\hbox{when }\nu\to+\infty,\qquad Q_\nu^\mu(\ch\tau)\sim e^{i\pi\mu}e^{-\tau/2}\sqrt{\pi\over2\,\sh\tau}\nu^{\mu-1/2} e^{-\tau\nu}
$$
$$
\hbox{when }\nu\to-\infty,\qquad Q_\nu^\mu(\ch\tau)\sim e^{i\pi\mu}e^{\tau/2}\sqrt{\pi\over2\,\sh\tau}(-\nu)^{\mu-1/2} e^{\tau\nu}.
$$
These equivalences are locally uniform with respect to the variable $\tau$, that is to say uniform on all interval $[\tau_0,\tau_1]$ with $0<\tau_0<\tau_1$.
\end{prop}
{\begin{proof}
If $\nu=n-{1\over2}$ with $n\in\N$ (see \cite{virchenko} page 48), we have
$$
P_\nu^\mu(\ch\tau)={\Gamma(\nu+1)\over\Gamma(\nu-\mu+1)}{1\over\sqrt{2\pi(\nu+1)\sh\tau}}\left[e^{\left(\nu+{1\over2}\right)\tau}+e^{-\pi i\left(\mu-{1\over2}\right)-\left(\nu+{1\over2}\right)\tau}\right]\left[1+\O\left({1\over\nu}\right)\right].
$$
A straightforward application of the Stirling formula shows that when $\nu\to+\infty$
$$
{\Gamma(\nu+1)\over\Gamma(\nu-\mu+1)}\sim{\sqrt{2\pi}\nu^{\nu+1/2}e^{-\nu}\over\sqrt{2\pi}(\nu-\mu)^{\nu-\mu+1/2}e^{-\nu+\mu}}
=\left({\nu\over\nu-\mu}\right)^{\nu+1/2}(\nu-\mu)^\mu e^{-\mu}
$$
$$
=(\nu-\mu)^\mu e^{-\mu}\exp\left(-\left(\nu+{1\over2}\right)\ln\left(1-{\mu\over\nu}\right)\right)\sim\nu^\mu
$$
consequently,
$$
P_\nu^\mu(\ch\tau)\sim\nu^\mu{1\over\sqrt{2\pi\nu\sh\tau}}{e^{\tau\over2}}e^{\tau\nu}={e^{\tau/2}\over\sqrt{2\pi\sh\tau}}\nu^{\mu-1/2}e^{\tau\nu},
$$
which gives us the first estimate. 

The second one is obtained directly thanks to the relation $P_\nu^\mu=P_{-\nu-1}^\mu$.

The third estimate follows directly from the formula (8.3) of \cite{virchenko} :
$$
Q_\nu^\mu(\ch\tau)\sim\sqrt{\pi\over2\,\sh\tau}\nu^{\mu-1/2}e^{i\pi\mu}e^{-\tau(\nu+1/2)}
$$
and the last estimation arises from the fact that for $\nu=n-{1\over2}$ with $n\in\Z$, we have
$$
Q_{-\nu-1}^\mu=Q_\nu^\mu.
$$
The locally uniform character of these equivalences come from explicits expressions of $P_\nu^\mu$ and $Q_\nu^\mu$ in terms of hypergeometric  functions (\cite{erdelyi1953}, tables pages 124-138) and estimations of these special functions (always locally uniform with respect to their parameters (\cite{virchenko}, pages 178-182). 
\end{proof}

\bibliographystyle{plain}
\bibliography{biblio}

\end{document}